%% file: main.tex
\documentclass[11pt,letterpaper,twoside,english]{amsart}
\usepackage{hyperref}
\usepackage[foot]{amsaddr}
\usepackage{amsmath}
\usepackage{amsthm}
\usepackage{amssymb}
\usepackage{tikz,pgfplots}
\usepackage{tikz-cd}
\usepackage{parselines}
\usepackage{booktabs}
\usepackage{mathrsfs}
\usepackage{setspace}
\usetikzlibrary{arrows.meta}

\usepackage[backend=biber,style=alphabetic,maxnames=5]{biblatex}
\newtheorem{theorem}{Theorem}[section]
\newtheorem{prop}[theorem]{Proposition}
\newtheorem{lemma}[theorem]{Lemma}

\newtheorem{definition}[theorem]{Definition}

\newtheorem{remark}[theorem]{Remark}

\newtheorem{def-thm}[theorem]{Definition-Theorem}
\newtheorem{claim}[theorem]{Claim}
\newtheorem{convention}[theorem]{Convention}

\numberwithin{equation}{section}
\begin{document}

\title{Space of nilpotent orbits and extension of period maps (I): The weight $3$ Calabi-Yau types}
\author{Haohua Deng}
\address{Department of mathematics at Duke University, 120 Science Drive, 117 Physics Building,
Campus Box 90320, Durham, North Carolina, 27708-0320}
\email{haohua.deng@duke.edu}
\date{}

\maketitle              


\begin{center}
\textbf{Abstract}
\end{center}
In Kato-Nakayama-Usui's theory, a certain space of nilpotent orbits can be constructed and serve as a completion of a given period map. This can be regarded as a generalization of Mumford's toroidal compactification for locally symmetric varieties. 

Kato-Nakayama-Usui's construction requires the existence of a weak fan, which is not known in general for non-classical cases. In this paper, we show after some slight modifications, such weak fans exist for a large class of period maps of weight $3$ Calabi-Yau type. In particular, for these cases a Kato-Nakayama-Usui type completion can be constructed.

\tableofcontents
\normalsize

\input{Sec1_Intro}
\input{Sec2_Prelim}
\input{Sec3_KUspace}
\input{Sec4_classical}
\input{Sec5_CY3}
\input{Sec6_Example}
\input{Sec7_weakfan}

\input{Sec8_appendix}

%
%

\printbibliography
\end{document}

%% file: Sec1_Intro.tex
\section{Introduction}\label{introsection}
\subsection{Overview}
Theories for compactifying locally symmetric varieties have been developed for decades and are now very rich. For example, in Satake-Baily-Borel's theory \cite{Sat60} and \cite{BB66}, for a locally symmetric variety $\Gamma\backslash D$ they constructed a projective but highly singular compact model by adding cusps along the boundary. This is known as the Satake-Baily-Borel compactification $\overline{\Gamma\backslash D}^{\mathrm{BB}}$. To improve the quality of the compactification, Mumford et. al. introduced the theory of toroidal compactification in \cite{AMRT10} which is not canonical in general, but will provide a smooth projective birational model for $\Gamma\backslash D$ with a decent choice of $\Gamma$-admissable polyhedral decomposition.

As an application of theories on compactifying locally symmetric varietes, Completing period mappings for which the Hodge varieties $\Gamma\backslash D$ are locally symmetric has also been well-investigated. These are called period maps of classical types in Hodge theory, which have been widely used to study the moduli space of polarized abelian varieties (weight $1$ case) and polarized K3 surfaces (weight $2$ case with $h^{2,0}=1$). 

One of the major obstructions of using Hodge theory to study families of general algebraic varieties, especially the degenerating behaviors, is the absence of theories completing period maps of non-classical types. Some recent works, for example, \cite{BKT20}, \cite{BBT22}, showed some critical abstract properties of general period maps, but we still need theories suitable for applications in other fields like algebraic geometry or mirror symmetry.

In recent years, Kato-Usui, or Kato-Nakayama-Usui have developed new theories of extending period maps of general types in \cite{KU08}, \cite{KNU10}. Their theory can be regarded as a generalization of Mumford's toroidal compacitification, but the resulting space is essentially different from the classical cases. Generally speaking, for a period mapping $S\xrightarrow[]{\varphi}\Gamma\backslash D$, the resulting enlarged space, which is written as $\Gamma\backslash D_{\Sigma}$, lies in a certain category called logarithmic manifolds, and its construction depends on a choice of a fan, or weak fan $\Sigma$ with certain distinguished properties. This enlarged space is called the space of nilpotent orbits or logarithmic polarized variation of Hodge structures.

One of the critical ingredients, also technical difficulties to apply Kato-Nakayama-Usui's theory, is to construct, or even prove the existence of $\Sigma$. The construction of $\Sigma$ for single-variable period maps is evident. In classical cases, Kato-Usui's theory agree with Mumford's toroidal compactification and $\Sigma$ can be chosen as \cite{AMRT10} using polyhedral reduction theory. Besides these known cases, the existence of such an object is an open problem in general. Some recent works, for example \cite{Den22} and \cite{Chen23} have shown the existence of such a fan or weak fan beyond trivial and classical cases.

\subsection{Purpose and organization of the paper}

The purpose of this paper is to show that Kato-Nakayama-Usui's theory can be applied to a large class of non-classical (multivariable) period maps. The main result can be summarized as follows: \\

\noindent \textbf{Main Result (Theorem \ref{formalmainthm}):} \textit{For a period map $\varphi: S\rightarrow \Gamma\backslash D$ of weight $3$ Calabi-Yau type, if along the boundary divisors of $S$ we only acquire type I or type IV degenerations in the sense of \cite[Example 5.8]{KPR19}, then there exists a finite modification of the period map, called $\tilde{\varphi}: \tilde{S}\rightarrow \Gamma\backslash D$, such that this modified period map admits a Kato-Usui type extension.}\\

A weight $3$ Calabi-Yau type limiting mixed Hodge structure $(W(N), F^{\bullet}, N)$ is called type I if:
\begin{equation}
    N\neq 0, \ N^2=0, \ h^{3,0}=1,
\end{equation}
or called type IV if:
\begin{equation}
    N^3\neq 0, N^4=0.
\end{equation}
They can also be characterized by their Hodge-Deligne diagrams, see Figure \ref{typeonelmhs} and Figure \ref{typeIVlmhs}. 

The meaning of finite modification will be explained in Section 6. Generally speaking, it suffices to prove the existence of $\Sigma$ which can provide such an extension. Indeed, $\Sigma$ can be obtained by taking all local monodromy nilpotent cones arising from the period map $\varphi$, then proceed a finitely generated subdivision. One of the central observations is the finiteness condition on subdivisions we need can be induced from the Siegel properties in the theory of arithmetic groups.

Our methods needs a slight modification on Kato-Nakayama-Usui's theory. Indeed, the object $\Sigma$ we will obtain is neither a fan in the sense of \cite{KU08} nor a weak fan in the sense of \cite{KNU10}, but a weak fan of restricted type. This modified object will provide a completion which is slighlty different from Kato-Usui's original theorems. For example, the resulting Kato-Usui type space constructed from a weak fan of restricted type is in general not a logarithmic manifold, but only a locally analytically constructible space (Definition \ref{deflocallyanalyticallyconstr}).

In Section 2, we review some necessary Hodge theory background. In Section 3 a summary of Kato-Usui's theory with the minimum requirements will be presented. In Section 4, we will show how Kato-Nakayama-Usui's theory works for the classical weight $1$ case. In Section 5, we will show the main result: For type I and type IV degenerations in the weight $3$ Calabi-Yau case, the construction can be derived from the weight $1$ case, which proves our main result. In Section 6, we conclude the main result by explaining the relation between logarithmic modification and period mappings. We will also review the main results from \cite{Den22} as an example. In Section 7 we will take a closer look on the combinatorial properties of weak fan. Finally in the appendix, we will give a more concrete overview on Kato-Nakayama-Usui's theory. We will also establish the validity of our modified definition on weak fan of restricted types.

\bigskip

\noindent \textbf{Acknowledgement:} The author is grateful to Matt Kerr and Colleen Robles for sharing a lot of motivating ideas. The author also thanks Ben Bakker, Chongyao Chen, Charles Doran, Patricio Gallardo, Philip Griffiths, Kazuya Kato, Bruno Klingler, Yilong Zhang and many others for related discussions.

%% file: Sec2_Prelim.tex
\section{Preliminaries on Hodge theory}

Let $S$ be a smooth quasi-projective variety, $\bar{S}$ be projective with $\bar{S}\backslash S$ a simple normal crossing divisor. We assume there is a $\mathbb{Z}$-local system $\mathcal{V}$ on $S$ which induces a polarized variation of Hodge structures. This induces a period map:
\begin{equation}\label{periodmapgeneralform}
    \varphi: S\rightarrow \Gamma \backslash D
\end{equation}
 where $D$ is the classifying space of polarized Hodge structures defined on an integral lattice $H_{\mathbb{Z}}$ with integral polarization form $Q$. Moreover, denote $G=\mathrm{Aut}(H, Q)$, then $\Gamma\leq G_{\mathbb{Z}}$ is just the monodromy group of the local system $\mathcal{V}\rightarrow S$.

 Such an abstract model can be realized in algebraic geometry: Suppose 
 \begin{equation}
     \pi: \mathcal{X}\rightarrow S
 \end{equation}
is a smooth projective family where $\mathcal{X}$ is smooth and $\pi$ is a holomorphic proper submersion. Moreover suppose for any $s\in S$, the fiber $X_s:=\pi^{-1}(s)$ is a smooth projective variety of fixed dimension $\mathrm{dim}_{\mathbb{C}}X_s=l$. Denote $\mathbb{Z}_{\mathcal{X}}$ as the $\mathbb{Z}$-constant sheaf on $\mathcal{X}$, then 
\begin{equation}
    \mathcal{V}:=R^l\pi_*(\mathbb{Z}_{\mathcal{X}})
\end{equation}
is a $\mathbb{Z}$-local system over $S$ whose fiber at $s\in S$ is $H:=H^l(X_s, \mathbb{Z})$. Take the Hodge decomposition 
\begin{equation}
    H_{\mathbb{C}}=H^l(X_s, \mathbb{Z})\otimes \mathbb{C}=\oplus_{p+q=l}H^{p,q}(X_s)
\end{equation}
on each fiber into consideration, this gives a variation of Hodge structure of weight $l$ on $S$ as well as a period map of the form \eqref{periodmapgeneralform}.

Moreover for $0\leq p\leq l$, let $\mathcal{F}^p$ be the subsheaves of $\mathcal{V}\otimes O_S$ which is known as Hodge bundles. Under the Gauss-Manin connection $\nabla$ of $\mathcal{V}$, we have the Griffiths transversality condition:
\begin{equation}\label{griffithstrans}
    \nabla (\mathcal{F}^p) \subset \Omega^1_S\otimes \mathcal{F}^{p-1}, \ 0\leq p\leq l.
\end{equation}

In this paper we will try to extend $\varphi$ over $\bar{S}$ with image in a properly enlarged space of $\Gamma\backslash D$. We first consider the local model of period map, i.e., suppose the period map is defined over a product of puntured disks:
\begin{equation}\label{periodmaplocalform}
    \varphi: (\Delta^*)^n\rightarrow \Gamma \backslash D
\end{equation}
where the local monodromy group $\Gamma$ is abelian and generated by $T_1,...,T_n$ where $T_j$ is obtained by the analytic continuation over a loop around the $j$-th coordinate divisor. If locally $S$ is isomorphic to $(\Delta^*)^k\times \Delta^{n-k}$, we can normalize it to the case $(\Delta^*)^n$ by setting $T_j=\mathrm{Id}$ for $j\geq k+1$. The following classical result is known as the Monodromy theorem.
\begin{theorem}[\cite{Gri70}, Theorem 3.1]
All $T_i$'s are quasi-unipotent, which means there exist positive integers $w_i, v_i$ such that $(T_i^{w_i}-I)^{v_i}= 0$ for every $i$. Moreover, $v_i\leq l+1$ for each $i$.
\end{theorem}
Consider the Jordan decomposition $T_j=T_j^{\mathrm{ss}}T_j^{\mathrm{u}}$, the monodromy theorem implies $T_j^{\mathrm{ss}}$ has finite order. Therefore after a base change we can assume all $T_j$'s are unipotent. 
Denote $\mathfrak{g}=\mathrm{Lie}(G)$ and $N_j=\log(T_j)$, then $N_j\in \mathfrak{g}_{\mathbb{Q}}$ are nilpotent elements defined over $\mathbb{Q}$, and $\{N_1,...,N_n\}$ generates a nilpotent abelian subalgebra of $\mathfrak{g}_{\mathbb{Q}}$. 

Consider the universal covering map:
\begin{equation}
\mathrm{exp}(2\pi i\bullet): \mathfrak{H}^n\rightarrow (\Delta^*)^n,
\end{equation}
where $\mathfrak{H}:=\{z\in \mathbb{C} \ | \ \mathfrak{Im}(z)>0\}$ is the upper-half plane. We have the following commutative diagram:
\begin{equation}
\begin{tikzcd}
\mathfrak{H}^n \arrow[d, "\mathrm{exp}(2\pi i\bullet)"] \arrow[r, "\tilde{\varphi}"] & D \arrow[d] \\
B=(\Delta^*)^n \arrow[r, "\varphi"] & \Gamma \backslash D
\end{tikzcd}
\end{equation}
and the action of $T_1,\ldots,T_n$ on $D$ can be interpreted as:
\begin{equation}
   \tilde{\varphi}(z_1,...,z_j+1,...,z_n)=T_j\tilde{\varphi}(z_1,...,z_n).
\end{equation}
Let $\Psi: \mathfrak{H}^n\rightarrow D$ and $\psi: (\Delta^*)^n\rightarrow D$ be defined as:
\begin{align}
    \Psi(z_1,...,z_n)&:= \mathrm{exp}(-\sum_{1\leq j\leq n}z_jN_j)\tilde{\varphi}(z_1,...,z_n),\\
    \psi(e^{2\pi iz_1},...,e^{2\pi iz_n})&:= \Psi(z_1,...,z_n).
\end{align}
Clearly $\psi$ is a well-defined map, and we have the following classical result which is known as Schmid's nilpotent orbit theorem.
\begin{theorem}[\cite{Sch73}, Sec. 4] 
The map $\psi$ can be extended holomorphically as:
\begin{equation}
    \psi: \Delta^n\rightarrow \check{D}
\end{equation}
where $\check{D}$ is the compact dual of $D$. Denote $\psi(0)=:F_0^{\bullet}\in \check{D}$, the following properties hold:
\begin{description}
\item[(i)] $N_jF_0^{p}\subset F_0^{p-1}$ for any $1\leq j\leq n$ and $0\leq p\leq l$;
\item[(ii)] $\mathrm{exp}(\sum_{j=1}^{n}z_jN_j)F_0^{\bullet}\in D$ for $z_j\in \mathbb{C}$ and $\mathfrak{Im}(z_j)>>0$ for any $j$.
\end{description}
\end{theorem}
In general, we have the following definition:
\begin{definition}\label{definitionnilporbit}
Suppose $N_1,...,N_n\in \mathfrak{g}_{\mathbb{Q}}$ are commuting nilpotent elements, and $F^{\bullet}\in \check{D}$. We call $\sigma:=\sum_{j=1}^{n}\mathbb{Q}_{\geq 0}N_j$ a nilpotent cone\footnote{Strictly speaking, $\sigma$ should be defined as the convex hull of $\{N_1,...,N_n\}$. For simplicity we assume $\sigma$ is not degenerate, equivalently $N_1,...,N_n$ are linearly independent.} and $Z:=\exp(\sigma_{\mathbb{C}})F^{\bullet}\subset \check{D}$ (resp. $Z^{\#}:=\exp(i\sigma_{\mathbb{R}})F^{\bullet}\subset \check{D}$) a nilpotent orbit (resp. nilpotent $i$-orbit) if the following conditions hold:
\begin{description}
\item[(i)] $NF^{p}\subset F^{p-1}$ for any $N\in \sigma$ and $0\leq p\leq l$;
\item[(ii)] $\mathrm{exp}(\sum_{j=1}^{n}z_jN_j)F^{\bullet}\in D$ for $z_j\in \mathbb{C}$ and $\mathfrak{Im}(z_j)>>0$ for any $j$.
\end{description}
\end{definition}
\begin{remark}
    In Definition \ref{definitionnilporbit} (i) is known as the Griffiths transversality condition and (ii) is called the positivity condition.
\end{remark}
For any nilpotent element $N\in \mathfrak{g}_{\mathbb{Q}}$, there is a naturally associated Jacobson-Morozov weight filtration $W(N)_{\bullet}$ which is increasing and defined on $H_{\mathbb{Q}}$, see for example \cite[Sec. 2.4]{Rob15}. In particular, if the nilpotency index of $N$ is $\mu$, then $W(N)_{-\mu}=\mathrm{Img}(N^{\mu})$ and $W(N)_{\mu}=H_{\mathbb{Q}}$. In \cite{CK82} Cattani and Kaplan proved the following result: 
\begin{theorem}\label{nilpconeweightfiltration}
    If $(\sigma, F^{\bullet})$ is a nilpotent orbit, then there is a weight filtration $W_{\bullet}$ on $H_{\mathbb{Q}}$ such that $W_{\bullet}=W(N)_{\bullet}$ for any $N$ in the relative interior of $\sigma$.  
\end{theorem}
As a consequence, For a nilpotent cone $\sigma$, it make sense to define the weight filtration associated to $\sigma$ as $W(\sigma)_{\bullet}$. Notice if $\tau\leq \sigma$ is a proper face, $W(\sigma)_{\bullet}$ and $W(\tau)_{\bullet}$ are different in general. In the rest of this paper, we should use the following convention:
\begin{convention}
    By a convex polyhedral cone $\sigma$, we should mean $\sigma$ is relatively open. In other words, $\sigma=\overline{\sigma}\backslash \{\hbox{Union of all proper faces of }\sigma\}$.
\end{convention}

\begin{definition}
    Let $W_{\bullet}$ be an increasing (weight) filtration defined on $H_{\mathbb{Q}}$ and $F^{\bullet}\in \check{D}$. We say $(W_{\bullet}, F^{\bullet})$ is a mixed Hodge structure (of weight $l$) polarized by $N$ on $(H,Q)$ if:
  \begin{description}
    \item[(i)] $W_{\bullet} = W(N)_{\bullet}[-l]$.
    \item[(ii)] $N(F^p)\subset F^{p-1}$.
    \item[(iii)] Writing $P_{l+k}:=\mathrm{ker}\{N^{k+1}:\mathrm{Gr}^W_{l+k}\rightarrow\mathrm{Gr}^W_{l-k-2}\}$ and the non-degenerate bilinear form $Q_k(v,w):=Q(v,N^kw)$ on $P_{l+k}$, $F^{\bullet}$ induces a weight $l+k$ Hodge structure on $P_{l+k}$ polarized by $Q_k$.
  \end{description}
\end{definition}
The following theorem regarding the relation between nilpotent orbits and limiting mixed Hodge structures is classical, see for example \cite[Thm. 3.13]{CK89}
\begin{theorem}\label{nilporbittolmhs}
      $(\sigma, F^{\bullet})$ gives a nilpotent orbit if and only if $(W(\sigma)[-l]_{\bullet}, F^{\bullet})$ is a mixed Hodge structure polarized by any $N\in \sigma$.  
\end{theorem}
\begin{remark}
    The triple $(W(\sigma)[-l]_{\bullet}, F^{\bullet}, \sigma)$\footnote{In the rest of the paper, we will denote $W(\sigma)_{\bullet}$ as $W(\sigma)$ for convenience.} associated to a nilpotent orbit $(\sigma, F^{\bullet})$ is called the associated limiting mixed Hodge structure (LMHS).
\end{remark}

%% file: Sec3_KUspace.tex
\section{Weak fan of restricted type and Kato-Usui spaces}


In this section we briefly overview Kato-Usui's theory \cite{KU08}, and introduce a modification of the theory which is suitable for applications.

\subsection{Kato-Usui's main theorems}\label{KUmainthms}
Let $\Sigma$ be a collection of rational nilpotent cones in $\mathfrak{g}_{\mathbb{Q}}$, and $|\Sigma|$ be its support in $\mathfrak{g}_{\mathbb{Q}}$. For $\sigma\in \Sigma$, set $\Gamma(\sigma):=\Gamma\cap \mathrm{exp}(\sigma_{\mathbb{R}})$. This is a multiplicative monoid, let $\Gamma(\sigma)^{\mathrm{gp}}$ be its associated group.

\begin{definition}
We say $\Sigma$ is strongly compatible with $\Gamma\leq G_{\mathbb{Z}}$ if the following holds:
\begin{description}
\item[(i)] $\Sigma$ is closed under the action by $\mathrm{Ad}_{\Gamma}$;
\item[(ii)] $\sigma_{\mathbb{R}}= \mathbb{R}_{\geq 0}\langle \mathrm{log}(\Gamma(\sigma))\rangle$ for any $\sigma\in \Sigma$.
\end{description}
\end{definition}

\begin{definition}[\cite{KU08}, \cite{KNU10}]\label{deffanweakfan}
We say $\Sigma$ is a fan (resp. weak fan) if the following (i) and (ii) (resp. (i) and (ii')) hold:
\begin{description}
\item[(i)] $\Sigma$ is closed under taking faces;
\item[(ii)] For any $\sigma, \tau\in \Sigma$, either $\sigma\cap \tau=\emptyset$ or $\sigma=\tau$;
\item[(ii')] For any $\sigma, \tau\in \Sigma$, suppose $\sigma\cap \tau$ is not empty, and there exists $F^{\bullet}\in \check{D}$ such that both $(\sigma, F^{\bullet})$ and $(\tau, F^{\bullet})$ are nilpotent orbits, then $\sigma=\tau$.
\end{description}
\end{definition}

Now we suppose $\Sigma$ is a collection of rational nilpotent cones strongly compatible with $\Gamma$. Define the spaces of nilpotent orbits and nilpotent $i$-orbits:
\begin{align}\label{defspaceofnilpotentorbits}
    D_{\Sigma}:&=\{(\sigma, Z) \ | \ \sigma\in \Sigma, Z \ \text{is a} \ \sigma-\text{nilpotent orbit} \ \}\\
   \nonumber &=\sqcup_{\sigma\in \Sigma}B(\tau)
\end{align}
where
\begin{center}
$B(\sigma):=\{\text{$\sigma$-nilpotent orbits}\}$,
\end{center}
and
\begin{align}\label{defspaceofnilpotentiorbits}
    D_{\Sigma}^{\#}:&=\{(\sigma, Z^{\#}) \ | \ \sigma\in \Sigma, Z^{\#} \ \text{is a} \ \sigma-\text{nilpotent $i$-orbit} \ \}\\
   \nonumber &=\sqcup_{\sigma\in \Sigma}B^{\#}(\tau),
\end{align}
where
\begin{center}
$B^{\#}(\sigma):=\{\text{$\sigma$-nilpotent $i$-orbits}\}$.
\end{center}
For $\sigma\in \Sigma$, we define $D_{\sigma}$ and $D_{\sigma}^{\#}$ by regarding \{all faces of $\sigma$\} as a collection of rational nilpotent cones strongly compatible with $\Gamma(\sigma)^{\mathrm{gp}}$. Notice that $D$ embeds into both $D_{\sigma}$ and $D_{\sigma}^{\#}$ by $F^{\bullet}\rightarrow (\{0\}, F^{\bullet})$, and $\Gamma$ acts on $D_{\Sigma}$ by:
\begin{equation}\label{gammaactiononnilporbit}
    \gamma: (\sigma, F^{\bullet})\rightarrow (\mathrm{Ad}_{\gamma}\sigma, \gamma F^{\bullet}).
\end{equation}

The main result of \cite{KU08} can be summarized as follows.
\begin{theorem}[\cite{KU08}, Thm. A]\label{katousuimainthmA}
Suppose $\Sigma$ is a fan which is strongly compatible with $\Gamma\leq G_{\mathbb{Z}}$ which is assumed to be neat, then: 
 \begin{description}
 \item[(i)] $\Gamma \backslash D_{\Sigma}$ admits a structure of logarithmic manifold which is also Hausdorff.
 \item[(ii)] For any $\sigma\in \Sigma$, the natural map $\Gamma(\sigma)^{\mathrm{gp}}\backslash D_{\sigma}\rightarrow \Gamma\backslash D_{\Sigma}$ is locally a homeomorphism.
 \end{description}
\end{theorem}
 
 \begin{theorem}[\cite{KU08}, Thm. B]\label{katousuimainthmB}
Let $\Gamma$ and $\Sigma$ be the same as above. For any period map $\varphi: S\rightarrow \Gamma \backslash D$ over a quasi-projective variety $S$ whose monodromy group is contained in $\Gamma$, $\varphi$ extends uniquely to:
\begin{center}
    $\overline{\varphi}: \bar{S}\rightarrow \Gamma\backslash D_{\Sigma}$
\end{center}
with $\varphi$ a morphism between logarithmic manifolds.
\end{theorem}

\begin{remark}
\
\begin{description}
    \item[(i)] We refer readers to \cite{KU08} for the more general statement. The theorems stated here are special cases;
    \item[(ii)] In \cite{KNU10}, $\Sigma$ is assumed to be a $\Gamma$-strongly compatible weak fan (instead of a fan), and the same results still hold.
\end{description}
\end{remark}

Logarithmic manifold is defined in \cite[Sec. 3.5.7]{KU08}. Roughly speaking, it is locally the vanishing locus of a (finitely generated) ideal of logarithmic differential $1$-forms corresponds to some fs logarthmic analytic space. One can also understand it as a space locally stratified by complex analytic spaces without concerning the logarithmic structure.

We will give a more detailed introduction to Kato-Usui's theory as well as a sketch of proof for Theorem \ref{katousuimainthmA} in the appendix. 

\subsection{An alternative definition of fan}

Besides the space of $\sigma$-nilpotent orbits (resp. $i$-orbits) $B(\sigma)$ (resp. $B^{\#}(\sigma)$), we also define:
\begin{center}
    $\tilde{B}(\sigma):=\{F^{\bullet}\in \check{D} | (\sigma, F^{\bullet}) \ \hbox{is a nilpotent orbit}\}$.
\end{center}
By Theorem \ref{nilpconeweightfiltration} and Theorem \ref{nilporbittolmhs}, we have:
\begin{center}
    $\tilde{B}(\sigma)=\{F^{\bullet}\in \check{D} | (W(\sigma)[-l], F^{\bullet}, \sigma) \ \hbox{is a limiting mixed Hodge structure}\}$.
\end{center}
Regarding the structure of $\tilde{B}(\sigma)\subset \check{D}$, \cite[Chap. 7]{KP16} says for $F^{\bullet}\in \tilde{B}(\sigma)$:
\begin{equation}\label{tangentspaceofbn}
    T_{F^{\bullet}}\tilde{B}(\sigma)=\mathfrak{z}(\sigma)\cap T_{F^{\bullet}}\check{D}
\end{equation}
where
\begin{equation}
    \mathfrak{z}(\sigma):=\cap_{N\in \sigma_{\mathbb{R}}}\mathrm{Ker}(\mathrm{ad}_N)\subset \mathfrak{g}_{\mathbb{C}}.
\end{equation}
We also denote
\begin{equation}
    Z(\sigma):=\{g\in G_{\mathbb{C}}|\mathrm{Ad}(g)N=N, \ \forall N\in \sigma_{\mathbb{Q}}\}\subset G_{\mathbb{C}}.
\end{equation}
as the centralizer of $\sigma$. Indeed, $Z(\sigma)\subset P_{W(\sigma)}(\mathbb{C})$ where $P_{W(\sigma)}$ is the $\mathbb{Q}$-parabolic subgroup of $G_{\mathbb{Q}}$ stabilizing $W(\sigma)$. Clearly $Z(\sigma)$ acts locally transitively on $\tilde{B}(\sigma)$.

The global structure of $\tilde{B}(\sigma)$ is more complicated. Fix $F^{\bullet}\in \tilde{B}(\sigma)$, and let $\tilde{B}(\sigma)^{\circ}$ be the $Z(\sigma)$-orbit containing $F^{\bullet}$. \cite[Chap. 7]{KP16} says:
\begin{equation}\label{structureofbn}
    \tilde{B}(\sigma)^{\circ}\cong Z(\sigma)\cap(U_{W(\sigma)}(\mathbb{C})\rtimes P_{W(\sigma)}(\mathbb{R}))/Z(\sigma)\cap \mathrm{Stab}_{F^{\bullet}}G_{\mathbb{C}}.
\end{equation}
However as \cite[Sec. 1.4]{KP16} has pointed out, $\tilde{B}(\sigma)$ may not be connected, and may have more than one (but finitely many) $Z(\sigma)$-orbits. These orbits can be classified by the type of limiting mixed Hodge structures (LMHS) induced by the corresponding nilpotent orbit.


We give an example to illustrate the point. Suppose $D$ is the period domain parametrizing polarized Hodge structures of type $(1,a+b,a+b,1)$ on $(H_{\mathbb{Q}}, Q)$ (hence the dimension of $H$ is $2(1+a+b)$), then for a nilpotent element $N\in \mathfrak{g}_{\mathbb{Q}}$ such that
\begin{equation}
    N\neq 0,  \ N^2=0, \ \mathrm{rank}(N)=a,
\end{equation}
$\tilde{B}(N)$ may contain $2$ disjoint $Z(\sigma)$-orbits parametrizing the following two types of LMHS described by their Hodge-Deligne diagrams, see Figure \ref{lmhstypefigure}.
\begin{figure}\label{lmhstypefigure}
\centering

    \begin{tikzpicture}[scale=1]
    \draw[-,line width=1.0pt] (0,0) -- (0,3.5);
    \draw[-,line width=1.0pt] (0,0) -- (3.5,0);
    \fill (0,3) circle (2pt);
    \node at (-0.3,3) {$1$};
    
    \fill (1,1) circle (2pt);
    \node at (0.7,1) {$a$};

     \fill (1,2) circle (2pt);
    \node at (0.7,2.3) {$b$};
    
    \fill (3,0) circle (2pt);
    \node at (3,-0.3) {$1$};
    
    \fill (2,1) circle (2pt);
    \node at (2.3,0.7) {$b$};
    
    \fill (2,2) circle (2pt);
    \node at (2.3,2) {$a$};
    
    \draw [-stealth](1.8,1.8) -- (1.2,1.2) [line width = 1.0pt];
  
\end{tikzpicture}
$
\mspace{50mu}
$
\begin{tikzpicture}[scale=1]
    \draw[-,line width=1.0pt] (0,0) -- (0,3.5);
    \draw[-,line width=1.0pt] (0,0) -- (3.5,0);
    \fill (1,3) circle (2pt);
    \node at (1,3.3) {$1$};

    \fill (0,2) circle (2pt);
    \node at (-0.3,2) {$1$};
    
    \fill (1,1) circle (2pt);
    \node at (1,0.7) {$a-2$};

     \fill (1,2) circle (2pt);
    \node at (1,2.3) {$b+1$};
    
    \fill (3,1) circle (2pt);
    \node at (3.3,1) {$1$};

     \fill (2,0) circle (2pt);
     \node at (2,-0.3) {$1$};
    
    \fill (2,1) circle (2pt);
    \node at (2,0.7) {$b+1$};
    
    \fill (2,2) circle (2pt);
    \node at (2,2.3) {$a-2$};
    
    \draw [-stealth](1.8,1.8) -- (1.2,1.2) [line width = 1.0pt];
    \draw [-stealth](0.8,2.8) -- (0.2,2.2) [line width = 1.0pt];
    \draw [-stealth](2.8,0.8) -- (2.2,0.2) [line width = 1.0pt];
    \end{tikzpicture}
    \caption{Type I and II LMHS for $(1,a+b,a+b,1)$} 
\end{figure}
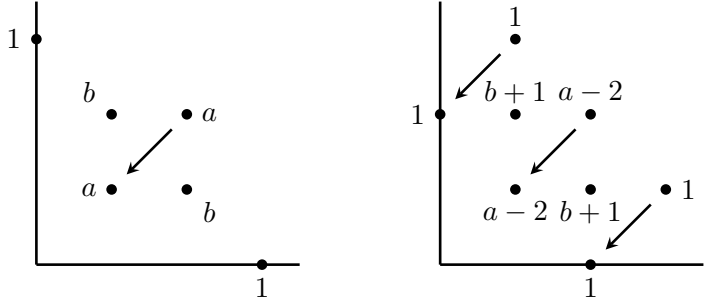
Here we use the type of LMHS defined by \cite[Chap. 5]{KPR19}. Particularly for the $(1,a+b,a+b,1)$ period domain, type $X$ means the Hodge numbers $\{h^{p,q}\}_{p,q\in \mathbb{Z}}$ satisfy $h^{4-X, 0}=1$. A quick way to distinguish these $2$ types of LMHS is checking whether $N$ annihilates $F^{3}$.
\begin{definition}
    Let $\Phi$ be an index set of limiting mixed Hodge structure types (corresponds to some fixed Hodge numbers $h^{p,q}$). We say a nilpotent orbit $(\sigma, F^{\bullet})$ is of type $\Phi$ if the LMHS induced by the nilpotent orbit has type included by $\Phi$.
\end{definition}
For example, in the example above, if a nilpotent orbit $(\sigma, F^{\bullet})$ gives type I LMHS, we say the nilpotent orbit is of type I. In this case if we let the index set $\Phi=\{\text{Pure LMHS}, \text{Type I LMHS}\}$, then $(\sigma, F^{\bullet})$ is of type $\Phi$.

Next we introduce alternative definitions to \eqref{defspaceofnilpotentorbits} and \eqref{defspaceofnilpotentiorbits}. Define:
\begin{align}
    D_{\Sigma, \Phi}:&=\{(\sigma, Z) \ | \ \sigma\in \Sigma, Z \ \text{is a} \ \sigma-\text{nilpotent orbit of type $\Phi$} \ \}\\
   \nonumber &=\bigsqcup_{\sigma\in \Sigma}B(\sigma, \Phi),
\end{align}
where
\begin{center}
$B(\sigma, \Phi):=\{\text{$\sigma$-nilpotent orbits of type $\Phi$}\}$,
\end{center}
and
\begin{align}
    D_{\Sigma, \Phi}^{\#}:&=\{(\sigma, Z^{\#}) \ | \ \sigma\in \Sigma, Z^{\#} \ \text{is a} \ \sigma-\text{nilpotent $i$-orbit of type $\Phi$} \ \}\\
   \nonumber &=\bigsqcup_{\sigma\in \Sigma}B^{\#}(\sigma, \Phi),
\end{align}
where
\begin{center}
$B^{\#}(\sigma, \Phi):=\{\text{$\sigma$-nilpotent $i$-orbits of type $\Phi$}\}$.
\end{center}
The action of $\Gamma$ on nilpotent orbits preserves the associated LMHS type, therefore the natural inclusions
\begin{align}
    D_{\Sigma, \Phi} &\hookrightarrow D_{\Sigma},\\
    \nonumber D^{\#}_{\Sigma, \Phi} &\hookrightarrow D^{\#}_{\Sigma}
\end{align}
descend to inclusions
\begin{align}
    \Gamma\backslash D_{\Sigma, \Phi} &\hookrightarrow \Gamma\backslash D_{\Sigma},\\
    \nonumber \Gamma\backslash D^{\#}_{\Sigma, \Phi} &\hookrightarrow \Gamma\backslash D^{\#}_{\Sigma}.
\end{align}
\begin{definition}\label{typeweakfan}
In Definition \ref{deffanweakfan}, if we keep condition (i) but replace condition (ii) or (ii') by the following:
\begin{description}
    \item[(ii'')]  For any $\sigma, \tau\in \Sigma$, suppose $(\sigma\cap \tau)^{\circ}$ is not empty, and there exists $F^{\bullet}\in \check{D}$ such that both $(\sigma, F^{\bullet})$ and $(\tau, F^{\bullet})$ are nilpotent orbits \underline{with type $\Phi$ LMHS}, then $\sigma=\tau$.
\end{description}
then we say $\Sigma$ is a type $\Phi$ weak fan. When there is no need to emphasize the set $\Phi$, we also say $\Sigma$ is a weak fan of restricted type.
\end{definition}

\begin{remark}
The condition (ii'') in Definition \ref{typeweakfan} is generally weaker than (ii) and (ii') in Definition \ref{deffanweakfan} unless $\Phi$ contains all possible LMHS types polarized by any $\sigma\in \Sigma$.
\end{remark}
Now suppose for a given period map
\begin{center}
$\varphi: S\rightarrow \Gamma \backslash D$,
\end{center}
$\Phi$ denotes all LMHS types arising via Schmid's nilpotent orbit theorem. We present the following variation of Kato-(Nakayama)-Usui's theorems. 

\begin{theorem}[Theorem A']\label{kumainthmA'}
Suppose $\Sigma$ is a type-$\Phi$ weak fan strongly compatible with $\Gamma\leq G_{\mathbb{Z}}$ which is assumed to be neat, then: 
 \begin{description}
 \item[(i)] $\Gamma \backslash D_{\Sigma, \Phi}$ admits a structure of locally analytically constructible space.
 \item[(ii)] For any $\sigma\in \Sigma$, the canonical projection $\Gamma(\sigma)^{\mathrm{gp}}\backslash D_{\sigma, \Phi}\rightarrow \Gamma\backslash D_{\Sigma, \Phi}$ is locally an isomorphism of locally analytically constructible spaces.
 \end{description}
\end{theorem}
 
 \begin{theorem}[Theorem B']\label{kumainthmB'}
Let $\Gamma$ and $\Sigma$ be the same as before. For a period map $\varphi: S\rightarrow \Gamma \backslash D$ over a quasi-projective variety $S$ whose monodromy group is contained in $\Gamma$. If all local monodromy nilpotent orbits along $\partial S$ have types contained in $\Phi$, then $\varphi$ extends uniquely to:
\begin{center}
    $\overline{\varphi}: \bar{S}\rightarrow \Gamma\backslash D_{\Sigma, \Phi}$
\end{center}
with $\overline{\varphi}$ a morphism between locally analytically constructible spaces.
\end{theorem}

Here we follow \cite[Sec. 3.1.4]{KU08} to define the analytically constructibility. Notice that a logarithmic manifold is locally analytically constructible, but the converse is not true in general.

\begin{definition}\label{deflocallyanalyticallyconstr}
    A subset $S$ of a complex analytic space $X$ is called analytically constructible if there exists a finite family $(A_j)_{j\in J}$ of closed analytic subspaces $A_j$ of $X$ and a closed analytic subspace $B_j\subset A_j$ for each $j$, called the strata of $S$, such that
    \begin{equation}
        S=\{x\in X \ | \ x\in A_j\Rightarrow x\in B_j\}.
    \end{equation}
    Suppose $S\subset A$ and $T\subset B$ are analytically constructible sets. A morphism $f: S\rightarrow T$ is a restriction of some analytic map $\hat{f}: A\rightarrow B$ which preserves the strata.
\end{definition}

The proof will be left to the appendix. The motivation for this definition is the existence of type $\Phi$ weak fan relies on $\Phi$. This feature is not captured by the original definitions in \cite{KU08} or \cite{KNU10}. Also, by passing to a weak fan of restricted types, we will not get a logarithmic manifold as the original Kato-Usui's theory in general.


%% file: Sec4_classical.tex
\section{Weak fan: The classical weight $1$ case}

\subsection{Basic settings} We consider the period domain $D_{(g,g)}$ parametrizing weight $1$ polarized Hodge structures on $(H_{\mathbb{Z}}, Q)$ with rank$(H_{\mathbb{Z}})=2g$. We have:
\begin{equation}
    D_{(g,g)}\cong \mathrm{Sp}(2g, \mathbb{R})/\mathrm{U}(g)
\end{equation}
Regarding the Hodge-theoretic degenerations, the limiting mixed Hodge structures $H_{\mathbb{C}}=\oplus_{p,q}H^{p,q}$ are characterized by $0\leq a=h^{1,1}\leq g$. 

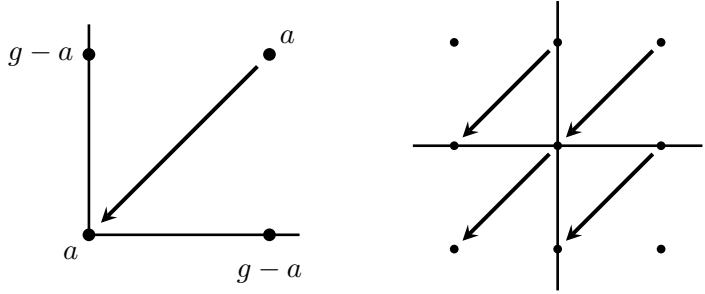
\begin{figure}\label{lmhsweight1}
\centering
    
    \begin{tikzpicture}[scale=0.8]
    \draw[-,line width=1pt] (0,0) -- (0,3.5);
    \draw[-,line width=1pt] (0,0) -- (3.5,0);
    \fill (0,3) circle (3pt);
    \node at (-0.8,3) {$g-a$};
    
    \fill (0,0) circle (3pt);
    \node at (-0.3,-0.3) {$a$};
    
    \fill (3,0) circle (3pt);
    \node at (3,-0.6) {$g-a$};
    
    \fill (3,3) circle (3pt);
    \node at (3.3,3.3) {$a$};
    
    \draw [-stealth](2.8,2.8) -- (0.2,0.2) [line width = 1.5pt];

    \end{tikzpicture}
    $
    \mspace{50mu}
    $
    \begin{tikzpicture}[scale=0.55]
    \draw[-,line width=1.0pt] (-3.5,0) -- (3.5,0);
    \draw[-,line width=1.0pt] (0,-3.5) -- (0,3.5);
    
    \fill (0,0) circle (3pt);
    \fill (2.5,0) circle (3pt);
    \fill (-2.5,0) circle (3pt);
    \fill (2.5,-2.5) circle (3pt);
    \fill (0,-2.5) circle (3pt);
    \fill (-2.5,-2.5) circle (3pt);
    \fill (2.5,2.5) circle (3pt);
    \fill (0,2.5) circle (3pt);
    \fill (-2.5,2.5) circle (3pt);
    
    \draw [-stealth](-0.2,2.3) -- (-2.3,0.2) [line width = 1.5pt];
    \draw [-stealth](2.3,2.3) -- (0.2,0.2) [line width = 1.5pt];
    
    \draw [-stealth](2.3,-0.2) -- (0.2,-2.3) [line width = 1.5pt];
    \draw [-stealth](-0.2,-0.2) -- (-2.3,-2.3) [line width = 1.5pt];

    \end{tikzpicture}
    \caption{Type $\mathrm{I}_a$ LMHS for $(g,g)$}
\end{figure}

Suppose $\varphi: S\rightarrow \Gamma\backslash D$ is a period map of type $(g,g)$ where $S$ is quasi-projective with normal crossing boundary divisors, and $\Gamma\leq \mathrm{Sp}(2g, \mathbb{Z})$ is assumed to be neat.

Let $\Sigma$ be the collection of all local monodromy nilpotent cones arising from $\varphi$. A description of $\Sigma$ is as follows: Since $\bar{S}$ is compact and $\bar{S}\backslash S$ is a normal crossing divisor, there exists a finite open cover $\{U_m\}_{1\leq m\leq M}$ such that $\cup_{1\leq m\leq M}U_s=\bar{S}$, and for each $1\leq m\leq M$, there exists $0\leq k_m\leq n$ such that
\begin{equation}
    U_m\cap S\cong (\Delta^{*})^{k_m}\times \Delta^{n-k_m}.
\end{equation}
Fix a base point $s_0\in S$, each of $U_m$ gives a $k_m$-dimensional monodromy nilpotent cone $\sigma_m$ and all of its faces, and for any point $s\in U_M\cap (\bar{S}\backslash S)$, by Schmid's nilpotent orbit theorem there is a nilpotent orbit associated to $s$ well-defined up to monodromy. 

By definition, $\Sigma$ is obtained by taking all $\mathrm{Ad}_{\Gamma}$-orbits of $\{\sigma_m \ | \ 1\leq m\leq M\}$ and all of their faces. Since changing the base point $s_0$ is equivalent to conjugating by an element in $G_{\mathbb{Q}}$, which does not change the topology of $\Sigma$, we may assume $\Sigma$ is canonically associated to the period map.

In this chapter we will prove the following theorem.
\begin{theorem}\label{mainthmclassical}
    There exists a polyhedral subdivision of $\Sigma$, denoted $\hat{\Sigma}$, such that $\hat{\Sigma}$ is a weak fan strongly compatible with $\Gamma$ and finitely generated under the action by $\mathrm{Ad}_{\Gamma}$. In other words, for any $\sigma, \tau\in \hat{\Sigma}$, $\tilde{B}(\sigma)\cap \tilde{B}(\tau)\neq \emptyset$ and $\sigma\cap \tau\neq \emptyset$ imply $\sigma=\tau$.
\end{theorem}

\begin{remark}
    In Section 7, we will see $\hat{\Sigma}$ constructed in Theorem \ref{mainthmclassical} is actually a fan. Moreover, by Proposition \ref{fanweakfanequivalent}, fan and weak fan are equivalent in the classical cases.
\end{remark}
We begin with analyzing cones in $\Sigma$ which could break the weak fan condition.
Let $\sigma, \tau\in \Sigma$ be two nilpotent cones. Suppose there exists a sequence $\{\gamma_i\}_{i\geq 0}\subset \Gamma$ such that there exists $0\neq N_i\in \sigma\cap \mathrm{Ad}_{\gamma_i}\tau$ and $\{F_i^{\bullet}\}\subset \check{D}$, we have
\begin{equation}\label{busterclassical}
    (\sigma, F_i^{\bullet}), (\mathrm{Ad}_{\gamma_i}\tau, F_i^{\bullet}), (N_i, F_i^{\bullet})
\end{equation}
are all nilpotent orbits. Theorem \ref{nilpconeweightfiltration} implies $\sigma$ and all $\mathrm{Ad}_{\gamma_i}\tau$ generate the same Jacobson-Morozov weight filtration $W(\sigma)$, therefore $\gamma_i\in \Gamma\cap P_{W(\sigma)}$ for any $i$. Moreover, we can assume $\gamma_0=\mathrm{Id}$, otherwise we shift the sequence $\{\gamma_i\}$ by $1$.

The following proposition can be regarded as the local version of Theorem \ref{mainthmclassical}:
\begin{prop}\label{finiteintersectionclassical}
For any triple \eqref{busterclassical} there exists a finite polyhedral decomposition $\Sigma(\sigma)$ of $\sigma$ such that $\sigma \cap \mathrm{Ad}_{\gamma_i}\tau$ is the union of some $\{\sigma_k\}\subset \Sigma(\sigma)$ for any $i$.
\end{prop}

We claim that to prove Proposition \ref{finiteintersectionclassical}, it suffices to prove the following: 
\begin{prop}\label{gammafiniteclass}
For the set $\{\gamma_i\}\in \Gamma$ given in Proposition \ref{uniformizationofbasepointHT}, its image under the projection
\begin{equation}
    g\in G_{\mathbb{C}} \rightarrow \overline{g}\in  Z(\sigma)\backslash G_{\mathbb{C}}/ Z(\tau)
\end{equation}
is \underline{finite}, where $Z(\sigma)\leq G_{\mathbb{C}}$ is the centralizer of $\sigma$.
\end{prop}
The reason is the intersection $\sigma \cap \mathrm{Ad}_{\gamma_i}\tau\subset \sigma$ is invariant under left multiplication on $\gamma_i$ by $Z(\sigma)$ and right multiplication on $\gamma_i$ by $Z(\tau)$, therefore the choice of subdivision $\Sigma(\sigma)$ is determined only by the class of $\gamma_i$ in $Z(\sigma)\backslash G_{\mathbb{C}}/Z(\tau)$, and we only have to consider a finite cell complex as a consequence of Proposition \ref{gammafiniteclass}. In Section 7 the way to obtain such a finite polyhedral decomposition will be explained.

We will prove Proposition \ref{finiteintersectionclassical} by studying different cases based on the type of LMHS generated by the triples \eqref{busterclassical}.

\subsection{The Hodge-Tate degeneration}\label{sectionHTdegenerationclassical}

In this subsection we shall consider the Hodge-Tate degenerations, i.e. In Figure \ref{lmhsweight1} we set $a=g$. 

Let $\Sigma_{\mathrm{HT}}\subset \Sigma$ be the subcollection generated by all local monodromy nilpotent cones arising from $\varphi$ whose corresponding LMHS are of Hodge-Tate type. Note that under the assumption cones in $\Sigma$ are relatively open, $\Sigma_{\mathrm{HT}}\subset \Sigma$ is not closed under taking faces in general. Now we should assume in \eqref{busterclassical}, $\sigma, \tau\in \Sigma_{\mathrm{HT}}$.

Next, we will prove the following proposition.
\begin{prop}\label{uniformizationofbasepointHT}
With the assumption $\sigma, \tau\in \Sigma_{\mathrm{HT}}$, in \eqref{busterclassical} we can choose all $F_i^{\bullet}=F_0^{\bullet}$. In other words,
\begin{center}\label{normalizedbusterclassicalHT}
    $(\sigma, F_0^{\bullet}), (\mathrm{Ad}_{\gamma_i}\tau, F_0^{\bullet}), (N_i, F_0^{\bullet})$
    \end{center}
    are all nilpotent orbits producing Hodge-Tate type LMHS.
\end{prop}
\begin{proof}
    Denote 
    \begin{align}
        H_{\mathbb{C}}&=\oplus H_i^{p,q}=H_i^{1,1}\oplus H_i^{0,0}, \\ 
        \nonumber \mathfrak{g}_{\mathbb{C}}&=\oplus \mathfrak{g}_i^{p,q}=\mathfrak{g}_i^{1,1}\oplus \mathfrak{g}_i^{0,0}\oplus \mathfrak{g}_i^{-1,-1}
    \end{align}
    as the Deligne splittings of the LMHS $(W(\sigma)[-1], F_i^{\bullet}, \sigma)$ and its adjoint. In this case \eqref{tangentspaceofbn} is read as: 
    \begin{equation}\label{localtangentspaceareequal}
        T_{F_i^{\bullet}}\tilde{B}(\sigma)= T_{F_i^{\bullet}}\tilde{B}(\mathrm{Ad}_{\gamma_i}\tau)= T_{F_i^{\bullet}}\tilde{B}(N_i)\cong \mathfrak{g}_i^{-1,-1}.
    \end{equation}
    We also notice that $F_0^{\bullet}\in \tilde{B}(\sigma)$. Therefore \eqref{structureofbn} and \eqref{localtangentspaceareequal} imply there exists $g_i\in Z(\sigma)\cap Z(\mathrm{Ad}_{\gamma_i}\tau)\cap Z(N_i)$ such that $F_0^{\bullet}=g_iF_i^{\bullet}$, hence we can replace $F_i^{\bullet}$ by $F_0^{\bullet}$ without destroying the properties of the triple \eqref{busterclassical}.
\end{proof}
From now on we denote:
   \begin{align}
        H_{\mathbb{C}}&=\oplus H^{p,q}=H^{1,1}\oplus H^{0,0}\\
       \nonumber \mathfrak{g}_{\mathbb{C}}&=\oplus \mathfrak{g}^{p,q}=\mathfrak{g}^{1,1}\oplus \mathfrak{g}^{0,0}\oplus \mathfrak{g}^{-1,-1}
    \end{align}
as the Hodge-Deligne splittings for $(W(\sigma)[-1], F_0^{\bullet}, \sigma)$ and its adjoint LMHS. Since for any $i$ we have $W(\mathrm{Ad}_{\gamma_i}\tau)=W(\sigma)$, we have $\mathrm{Ad}_{\gamma_i}\tau \subset \mathfrak{g}^{-1,-1}$. 


The remainder of this section will be dedicated to proving Proposition \ref{gammafiniteclass} for the case $\sigma, \tau\in \Sigma_{\mathrm{HT}}$.
Denote
\begin{align}
    &\mathcal{Z}(\sigma):=\mathrm{exp}(\sigma_{\mathbb{C}})F_0^{\bullet}=:(\sigma, F_0^{\bullet})\\
  \nonumber  &\mathcal{Z}^{\#}(\sigma):=\mathrm{exp}(i\sigma_{\mathbb{R}})F_0^{\bullet}=:(\sigma, F_0^{\bullet})^{\#}
\end{align}
as the nilpotent orbit and nilpotent $i$-orbit associated to $(W(\sigma)[-1], F_0^{\bullet}, \sigma)$. For different LMHS's we use the similar notation. The triplet of nilpotent orbit:
\begin{align}
    &\mathcal{Z}(\sigma):=(\sigma, F_0^{\bullet}),\\
   \nonumber &\mathcal{Z}_i(\tau):=(\mathrm{Ad}_{\gamma_i}\tau, F_0^{\bullet}),\\
   \nonumber &\mathcal{Z}_i(N_i):=(0\neq N_i\in \sigma\cap \mathrm{Ad}_{\gamma_i}\tau, F_0^{\bullet}) 
\end{align}
induces the triplet of nilpotent $i$-orbit:
\begin{align}\label{gangstatriple2}
    &\mathcal{Z}^{\#}(\sigma):=(\sigma, F_0^{\bullet})^{\#},\\
   \nonumber &\mathcal{Z}_i^{\#}(\tau):=(\mathrm{Ad}_{\gamma_i}\tau, F_0^{\bullet})^{\#},\\
   \nonumber &\mathcal{Z}_i^{\#}(N_i):=(0\neq N_i\in \sigma\cap \mathrm{Ad}_{\gamma_i}\tau, F_0^{\bullet})^{\#}   
\end{align}
which satisfies the third is contained in the intersection of the first two. 

\begin{prop}\label{openconvexcone}
There exist open convex polyhedral cones $\mathcal{C}(\sigma), \mathcal{C}(\tau)\subset \mathfrak{g}_{\mathbb{Q}}\cap\mathfrak{g}^{-1,-1}$ such that 
\begin{align}
    \sigma&\subset \mathcal{C}(\sigma),\\
   \nonumber \tau&\subset \mathcal{C}(\tau),\\
   \nonumber F_0^{\bullet}&\in \tilde{B}(\mathcal{C}(\sigma))\cap \tilde{B}(\mathcal{C}(\tau)).
\end{align}
\end{prop}
\begin{proof}
 For a fixed $F_0^{\bullet}$, the Griffiths transversality condition in Definition \ref{definitionnilporbit} is trivial in the classical case.
 Let $\mathcal{C}\subset \mathfrak{g}^{-1,-1}$ be the subset containing all $N\in \mathfrak{g}^{-1,-1}$ such that $(N, F_0^{\bullet})$ is a nilpotent orbit of Hodge-Tate type. Moreover,  $\mathcal{C}$ is open convex according to \cite[Chap. 2.1]{AMRT10}\footnote{$\mathcal{C}$ is a homogeneous self-adjoint cone in the sense of \cite[Chap. 2]{AMRT10}.}. It is possible to choose a set of vectors $\{v_i\}\subset \overline{\mathcal{C}}$ containing all dimension $1$ faces of $\sigma$, and whose convex hull is open in $\overline{\mathcal{C}}$, hence we can just take $\mathcal{C}(\sigma)$ to be the interior of this convex hull. $\mathcal{C}(\tau)$ can be constructed in the same way. 
\end{proof}
\begin{remark}
    Another way to see Proposition \ref{openconvexcone} is using the fact that in $\mathcal{C}$, polyhedral cones and Siegel sets are cofinal. See \cite[Chap. 2]{AMRT10}.
\end{remark}
Moreover, we may assume $\mathcal{C}(\sigma)$ and $\mathcal{C}(\tau)$ are both simplicial. Otherwise suppose $\mathcal{C}(\sigma)$ is not simplicial, we proceed a finite simplicial decomposition on $\mathcal{C}(\sigma)$\footnote{Such a simplicial decomposition can be chosen as a star subdivision introduced in \cite[Chap. 11]{CLS11}}. If the sequence $\{N_i\}$ in Proposition \ref{uniformizationofbasepointHT} is infinite, there exists an open simplicial cone $\mathcal{C}(\sigma)^{'}$ in the simplicial decomposition of $\mathcal{C}(\sigma)$ and an infinite subsequence of $\{N_i\}$ lie in $\mathcal{C}(\sigma)^{'}$, then we can replace $\sigma$ by $\sigma\cap \mathcal{C}(\sigma)^{'}$ without destroying the properties of triples in \eqref{busterclassical}. The argument for $\mathcal{C}(\tau)$ is similar.

Now we suppose the open simplicial cones $\mathcal{C}(\sigma), \mathcal{C}(\tau)\subset \mathfrak{g}_{\mathbb{Q}}\cap\mathfrak{g}^{-1,-1}$ are generated by:
\begin{align}
    \mathcal{C}(\sigma)&=\langle N_1,...,N_k\rangle_{\mathbb{Q}>0}\\
   \nonumber \mathcal{C}(\tau)&=\langle M_1,...,M_k\rangle_{\mathbb{Q}>0}.
\end{align}
with $k=\mathrm{dim}_{\mathbb{Q}}(\mathfrak{g}^{-1,-1})$. Denote
\begin{align}
    U(\mathcal{C}(\sigma))&:=\{\mathrm{exp}(\sum_{j=1}^k z_jN_j)F_0^{\bullet} \ | \ 0\leq \mathfrak{Re}(z_j)\leq 1\}\\
    \nonumber U_i(\mathcal{C}(\tau))&:=\{\mathrm{exp}(\sum_{j=1}^k z_j(\gamma_iM_j\gamma_i^{-1}))F_0^{\bullet} \ | \ 0\leq \mathfrak{Re}(z_j)\leq 1\}
\end{align}
as the neighborhood of $\mathcal{Z}^{\#}(\mathcal{C}(\sigma))$ (resp. $\mathcal{Z}^{\#}_i(\mathcal{C}(\tau))$) in $\mathcal{Z}(\mathcal{C}(\sigma))$ (resp. $\mathcal{Z}_i(\mathcal{C}(\tau))$) bounded in the real directions. Clearly $\mathcal{Z}^{\#}_i(N_i)\subset U(\mathcal{C}(\sigma))$. 


\begin{lemma}
There exists 
\begin{equation}
    g_i\in \mathrm{Sp}(2g, \mathbb{Z})\cap \mathrm{exp}(\mathfrak{g}^{-1,-1})
\end{equation}
such that
\begin{equation}
    \mathcal{Z}_i^{\#}(N_i)\subset U(\mathcal{C}(\sigma))\cap g_iU_0(\mathcal{C}(\tau)).
\end{equation}
\end{lemma}
\begin{proof}
According to the assumptions, we have:
\begin{equation}
    F_0^{\bullet}, \ \gamma_iF_0^{\bullet}\in \tilde{B}(\mathrm{Ad}_{\gamma_i}\tau).
\end{equation}
Since $\mathrm{Ad}_{\gamma_i}\mathcal{C}(\tau)$ is open in $\mathfrak{g}^{-1,-1}=T_{\gamma_iF_0^{\bullet}}\tilde{B}(\mathrm{Ad}_{\gamma_i}\tau)$ (see \eqref{localtangentspaceareequal}), we can assume 
\begin{equation}
    F_0^{\bullet}=\mathrm{exp}(\sum_{j=1}^{k} z^i_j(\gamma_iM_j\gamma_i^{-1}))\gamma_iF_0^{\bullet}
\end{equation}
for some $\{z^i_j\}\subset \mathbb{C}$, hence we can just take 
\begin{equation}
    \mu_i:=\mathrm{exp}(-\sum_{j=1}^{k} \lfloor\mathfrak{Re}(z^i_j)\rfloor (\gamma_iM_j\gamma_i^{-1}))\in \mathrm{Sp}(2g, \mathbb{Z})\cap \mathrm{exp}(\mathfrak{g}^{-1,-1}),
\end{equation}
if we set 
\begin{equation}
    g_i:=\mu_i\gamma_i,
\end{equation}
since $\mu_i\in \mathrm{exp}(\mathfrak{g}^{-1,-1})$ centralizes $\mathfrak{g}^{-1,-1}$, we have 
\begin{equation}
    F_0^{\bullet}=\mathrm{exp}(\sum_{j=1}^{k} w^i_j(\gamma_iM_j\gamma_i^{-1}))g_iF_0^{\bullet}
\end{equation}
with $\{w^i_j\}\subset \mathbb{C}$ and $0\leq \mathfrak{Re}(w^i_j)\leq 1$, this implies we have
\begin{equation}
   \mathcal{Z}^{\#}_i(N_i)\subset g_iU_0(\mathcal{C}(\tau))
\end{equation}
as $F_0^{\bullet}\in g_iU_0(\mathcal{C}(\tau))$ and $N_i\in \mathrm{Ad}_{\gamma_i}\tau\subset \mathrm{Ad}_{\gamma_i}\mathcal{C}(\tau)$.
\end{proof}
To sum up, we must have for any $i$,
\begin{equation}\label{intersection2}
    U(\mathcal{C}(\sigma))\cap g_iU_0(\mathcal{C}(\tau))\neq \emptyset
\end{equation}
and it contains a nilpotent $i$-orbit (i.e., $(N_i, F_0^{\bullet})^{\#}$).

On the other hand, for any $K>0$, define:
\begin{align}
    U(\mathcal{C}(\sigma))_{>K}&:=\{\mathrm{exp}(\sum_{j=1}^{k} z^i_jN_j)F_0^{\bullet} \ | \ 0\leq \mathfrak{Re}(z^i_j)\leq 1, \mathfrak{Im}(z^i_j)>K\},\\
    \nonumber U_0(\mathcal{C}(\tau))_{>K}&:=\{\mathrm{exp}(\sum_{j=1}^{k} z^i_jM_j)F_0^{\bullet} \ | \ 0\leq \mathfrak{Re}(z^i_j)\leq 1, \mathfrak{Im}(z^i_j)>K\}.
\end{align}
\normalsize
We fix a $K_0>0$ such that $U(\mathcal{C}(\sigma))_{>K_0}$, $U_0(\mathcal{C}(\tau))_{>K_0}\subset D$. We know from \cite[Thm. 1.5]{BKT20} for any $\epsilon>0$, there exist two finite collections of Siegel sets $\{\mathfrak{S}_p\}$, $\{\mathfrak{T}_q\}$ of $D$
such that 
\begin{align}
    U(\mathcal{C}(\sigma))_{> K_0+\epsilon}&\subset \cup \mathfrak{S}_p,\\
    \nonumber U_0(\mathcal{C}(\tau))_{> K_0+\epsilon}&\subset \cup \mathfrak{T}_q.
\end{align}
Also by the basic Siegel properties\footnote{See for example, \cite[Prop. III.2.19]{BJ06}}, for any arithmetic subgroup $\Gamma\leq \mathrm{Sp}(2g,\mathbb{Q})$, the set 
\begin{equation}\label{siegelpropertyarithmeticgroup}
    S(\sigma, \tau):=\{\gamma\in \Gamma|(\cup \mathfrak{S}_p)\cap (\cup \gamma\mathfrak{T}_q)\neq \emptyset\}
\end{equation}
is finite. Since the intersection in \eqref{intersection2} contains a non-trivial nilpotent $i$-orbit $\mathcal{Z}^{\#}_i(N_i)$, combine with the evident fact that
\begin{align}
    \mathcal{Z}^{\#}_i(N_i)&\cap U(\mathcal{C}(\sigma))_{> K_0+\epsilon}\neq \emptyset\\
   \nonumber \mathcal{Z}^{\#}_i(N_i)&\cap g_iU_0(\mathcal{C}(\tau))_{> K_0+\epsilon}\neq \emptyset,
\end{align}
we conclude the sequence $\{g_i=\mu_i\gamma_i\}$ must lie in the set $S(\sigma, \tau)$ thus contain only finitely many different elements, therefore $\{\gamma_i\}$ only gives finitely many $Z(\sigma)$-right cosets (as $\gamma_i=\mu_i^{-1}g_i$ and all $\mu_i\in Z(\sigma)$). To sum up, we get the following proposition, which is stronger than Proposition \ref{gammafiniteclass}:
\begin{prop}\label{gammafiniteclassstronger}
For the set $\{\gamma_i\}\in \Gamma$ given in Proposition \ref{uniformizationofbasepointHT}, its image under the projection
\begin{equation}
    g\in G_{\mathbb{C}} \rightarrow \overline{g}\in  Z(\sigma)\backslash G_{\mathbb{C}}
\end{equation}
is \underline{finite}, where $Z(\sigma)\leq G_{\mathbb{C}}$ is the centralizer of $\sigma$.
\end{prop}
\begin{remark}
    The necessity of stating Proposition \ref{gammafiniteclass} with $(Z(\sigma), Z(\tau))$-double cosets instead of only $Z(\sigma)$-right cosets will be clear in future works where we will treat more general cases.
\end{remark}

\subsection{Arbitrary degenerations}\label{sectionallclassicaldegenerations}

We will generalize results in the last subsection to all cones in $\Sigma$. The strategy is to translate general cases into the Hodge-Tate case. 

Like before, let $\sigma, \tau\in \Sigma$ be two monodromy nilpotent cones. Any triples in \eqref{busterclassical} are now producing type $\mathrm{I}_a$ LMHS with $0\leq a\leq g$, when $a=g$ this is the Hodge-Tate degeneration, see Figure \ref{lmhsweight1}, and again we assume $\gamma_0=\mathrm{Id}$. We should now prove Proposition \ref{finiteintersectionclassical} for any triples in \eqref{busterclassical} regardless of the associated LMHS types.

\begin{prop}\label{uniformizationofbasepoint}
Without the restriction of Hodge-Tate LMHS, in \eqref{busterclassical} we can still choose all $F_i^{\bullet}=F_0^{\bullet}$. In other words,
\begin{center}\label{normalizedbusterclassical}
    $(\sigma, F_0^{\bullet}), (\mathrm{Ad}_{\gamma_i}\tau, F_0^{\bullet}), (N_i, F_0^{\bullet})$
    \end{center}
    are all nilpotent orbits.
\end{prop}
\begin{proof}
    An analog of Proposition \ref{uniformizationofbasepointHT}. 
\end{proof}

Now the splitting of $(W(\sigma)[-1], F_0^{\bullet}, \sigma)$ becomes:
\begin{align}
        H_{\mathbb{C}}&=H^{1,1}\oplus H^{1,0}\oplus H^{0,1}\oplus H^{0,0},\\
        \nonumber \mathfrak{g}_{\mathbb{C}}&=\oplus_{-1\leq p,q\leq 1} \mathfrak{g}^{p,q}.
 \end{align}
with $\sigma$, $\mathrm{Ad}_{\gamma_i}\tau\subset \mathfrak{g}^{-1,-1}$. Consider the splitting:
\begin{equation}
    H_{\mathbb{C}}=H_1\oplus H_2
\end{equation}
where
\begin{align}
    H_1&:=H^{1,0}\oplus H^{0,1};\\
   \nonumber H_2&:=H^{0,0}\oplus H^{1,1}.
\end{align}
Clearly $H_2$ is defined over $\mathbb{R}$ and $\mathfrak{g}^{-1,-1}$ acts on $H_1$ trivially. For any $N\in \mathfrak{g}^{-1,-1}$ denote $\bar{N}\in \mathrm{End}(H_2)$ as an nilpotent endomorphism of $H_2$.

We use matrix representations to illustrate. Let
\begin{equation}\label{intbasisforclassical}
\mathcal{B}:=\{e_1,...,e_a,f_1,...,f_{h-a},f^{h-a},...,f^1,e^a,...,e^1\}
\end{equation}
be an integral basis of $(H_{\mathbb{Z}}, Q)$ such that:
\begin{align}
    &W(\sigma)_{-1}\cong \mathrm{span}\{e_1,...,e_a\};\\
   \nonumber &W(\sigma)_0\cong \mathrm{span}\{e_1,...,e_a, f_1,f^1,...,f_{g-a},f^{g-a}\};\\
   \nonumber  &W(\sigma)_1\cong H_{\mathbb{Q}},
\end{align}
and the alternating polarization form $Q$ satisfies:
\begin{align}
    &Q(e^i,e_j)=\delta^i_j, 1\leq i,j\leq a\\
  \nonumber  &Q(f^i,f_j)=\delta^i_j, 1\leq i,j\leq g-a.
\end{align}
and $Q(\alpha, \beta)=0$ for other cases when $\alpha, \beta\in \mathcal{B}$. Such a basis always exists, if we are allowed to replace $G_{\mathbb{Z}}$ by an arithmetic subgroup of $G_{\mathbb{Q}}$.

Under the ordered basis \eqref{intbasisforclassical}, We have:

\begin{equation}\label{polarizationformclassical}
    Q=\left[\begin{array}{c|c|c|c}
& & &-E_a\\ \hline

& &-E_{g-a}&\\ \hline

 &E_{g-a} & &\\ \hline
E_a& & &\\
\end{array}\right]
\end{equation}
where $E_n$ is the $n\times n$ matrix defined by
$\begin{pmatrix}
    & & 1\\
    & ... & \\
    1 & &
\end{pmatrix}$
, and 
\begin{equation}\label{matrixformofnandgamma}
    N=
\left[\begin{array}{c|c|c|c}
& & & *\\ \hline

& & &\\ \hline

 & & &\\ \hline
& & &\\
\end{array}\right],\;\;\;
\gamma_i=
\left[\begin{array}{c|c|c|c}
A& * & *& B\\ \hline

& *& *& *\\ \hline

 & *&* &*\\ \hline
& & &A^*\\
\end{array}\right].
\end{equation}
for any $N\in \mathfrak{g}^{-1,-1}$ and $\gamma_i$ in the triples provided by Proposition \ref{uniformizationofbasepoint}. The form of $\gamma_i$ is obtained from the evident fact $\gamma_i\in \Gamma\cap P_{W(\sigma)}$. With respect to the basis $\mathcal{B}$, consider the corresponding $\mathbb{Q}$-Levi decomposition:
\begin{equation}
    P_{W(\sigma)}=L_{W(\sigma)}\ltimes  U_{W(\sigma)},
\end{equation}
hence for any $\gamma_i$ we have the unique Levi decomposition:
\begin{equation}
    \gamma_i=\gamma_i^{l}\cdot \gamma_i^u,
\end{equation}
in which $\gamma_i^l, \gamma_i^u$ have the matrix forms as:

\begin{equation}
    \gamma_i^l=
\left[\begin{array}{c|c|c|c}
A&  & & \\ \hline

& *&* & \\ \hline

 & *&* &\\ \hline
& & &A^*\\
\end{array}\right],\;\;\;
\gamma_i^u=
\left[\begin{array}{c|c|c|c}
1& * & *& B\\ \hline

& 1& & *\\ \hline

 & &1 &*\\ \hline
& & &1\\
\end{array}\right].
\end{equation}
Note that it is clear we have:
\begin{align}
    &\mathrm{Ad}_{\gamma_i^u}N=N\\
   \nonumber &\mathrm{Ad}_{\gamma_i}N=\mathrm{Ad}_{\gamma_i^l}N.
\end{align}

Let $\tilde{Q}:=Q|_{H_2}$ be the induced polarization form on the $\mathbb{R}$-vector space $H_2$. We will associate $(H_2,\tilde{Q})$ an integral structure as follows. There is a ($\tilde{Q}$-isotropic) basis for $H^{2,2}$ of the following form:
\begin{equation}\label{standardbasisH2}
    \alpha_i=e^i+g_i+\sum_{j=1}^{a}z_{ij}e_j, 1\leq i\leq a.
\end{equation}
where $g_i\in \mathrm{span}\{f_k,f^k\}_{\mathbb{R}}$ and $z_{ij}\in \mathbb{C}$. 
\begin{prop}\label{integralstructure}
    The set $\mathcal{E}:=\{\beta_i:=\frac{\alpha_i+\bar{\alpha_i}}{2}, e_j\}_{1\leq i,j\leq a}$ is a basis of $(H_2)_{\mathbb{R}}$, and $\tilde{Q}$ is integral under the basis $\mathcal{E}$.
\end{prop}
\begin{proof}
    $\mathcal{E}$'s realness as well as $\tilde{Q}(e^i, \beta_j)=\delta^i_j$ is clear. We only have to check the isotropic condition $\tilde{Q}(\beta_i,\beta_j)=0$.
    We have:
    \begin{align}
        0&=\tilde{Q}(\alpha_i,\alpha_j)\\
        \nonumber &=\tilde{Q}(e^i+g_i+\sum_{k=1}^{a}z_{ik}e_k, e^j+g_j+\sum_{k=1}^{a}z_{jk}e_k)\\
       \nonumber &=z_{ij}-z_{ji}+\tilde{Q}(g_i,g_j).
    \end{align}
    Therefore $z_{ij}-z_{ji}=-\tilde{Q}(g_i,g_j)$ is real. Moreover we have:
    \begin{align}
        0&=z_{ij}-z_{ji}+\tilde{Q}(g_i,g_j)\\
       \nonumber &=\frac{z_{ij}+\bar{z_{ij}}}{2}-\frac{z_{ji}+\bar{z_{ji}}}{2}+\tilde{Q}(g_i,g_j)\\
       \nonumber &=\tilde{Q}(e^i+g_i+\sum_{k=1}^{a}\frac{z_{ik}+\bar{z_{ik}}}{2}e_k, e^j+g_j+\sum_{k=1}^{a}\frac{z_{jk}+\bar{z_{jk}}}{2}e_k)\\
      \nonumber  &=\tilde{Q}(\beta_i, \beta_j).
    \end{align}
\end{proof} 
\begin{remark}
     As a consequence of Proposition \ref{integralstructure}, $\mathbb{Z}\mathcal{E}$ realizes $((H_2)_{\mathbb{Z}}, \tilde{Q})$ as an integral symplectic lattice, under which $\tilde{Q}$ has the matrix form:
\begin{equation}\label{inducedformforclassical}
    \tilde{Q}:=Q|_{H_2}=\left[\begin{array}{c|c}
 & E_a\\ \hline

 -E_a&  
\end{array}\right].
\end{equation}
\end{remark}
\noindent Denote 
\begin{equation}
    H_2^{'}:=\mathrm{span}_{\mathbb{Z}}\{e_1,...,e_a,e^a,...,e^1\}\subset H_{\mathbb{Z}},
\end{equation}
there is a natural isomorphism of integral symplectic lattices $(H_2^{'}, Q|_{H_2^{'}})\xrightarrow{\sim} ((H_2)_{\mathbb{Z}}, \tilde{Q})$ given by 
\begin{equation}\label{identificationintegrallattice}
    e_i\rightarrow e_i, \ e^i\rightarrow \beta_i=e^i+g_i+\sum_{j=1}^{a}\mathfrak{Re}(z_{ij})e_j, \ 1\leq i\leq a.
\end{equation}

Under this identification, for any $N\in \mathfrak{g}^{-1,-1}$ and $\gamma_i^l\in \Gamma_L:=\Gamma\cap L_{W(\sigma)}$, both viewed as operators on $(H_2^{'})_{\mathbb{Q}}$, we have the induced $\tilde{N},\tilde{\gamma_i^l}$ as operators on $H_2$ by $N, \gamma_i^l$. If $N$ and $\gamma$ has the matrix forms given by \eqref{matrixformofnandgamma}, under the integral structure $((H_2)_{\mathbb{Z}}, \tilde{Q})$ and basis $\mathcal{E}$, $\tilde{N}$ and $\tilde{\gamma_i^l}$ has the matrix form:
\begin{equation}
    \tilde{N}=
\left[\begin{array}{c|c}
 & B\\ \hline

 &  
\end{array}\right],\;\;\;
\tilde{\gamma_i^l}=
\left[\begin{array}{c|c}
A & \\ \hline

 & A^*
\end{array}\right].
\end{equation}
Clearly under this identification, $\tilde{N}\in \mathfrak{sp}(H_2,\mathbb{Q})$ and $\tilde{\gamma_i^l}\in \mathrm{Sp}(H_2,\mathbb{Z})$\footnote{Strictly speaking, this should be an arithmetic subgroup of $\mathrm{Sp}(H_2, \mathbb{Q})$, but we still write it as the $\mathbb{Z}$-group for convenience.}.

We now turn to the Hodge filtration $F_0^{\bullet}$. On $H$ we have:
\begin{equation}
    F_0^{\bullet}(p)=\oplus_{s\geq p}H^{s,q}.
\end{equation}
Denote the induced Hodge filtration on $H_2$ as:
\begin{align}
    &\tilde{F}_0^2:=\mathrm{img}\{H^{1,1}\hookrightarrow H_2\},\\
    \nonumber &\tilde{F}_0^1:=H_{2, \mathbb{C}}.
\end{align}
Which can be identified as an element in $\check{D}_{(a,a)}$, the compact dual of the period domain $D_{(a,a)}$ parametrizing weight $1$ polarized $\mathbb{Z}$-Hodge structures on $((H_2)_{\mathbb{Z}}, \tilde{Q})$. 

The weight filtration $W(\sigma)$ on $H$ gives a natural weight filtration $W(\tilde{\sigma})$ on $H_{2, \mathbb{Q}}$ by:
\begin{align}
    W(\tilde{\sigma})_{-1}=W(\tilde{\sigma})_0&=\mathrm{span}_{\mathbb{Q}}\{e_i\}_{1\leq i\leq a}\\
    \nonumber W(\tilde{\sigma})_1&=(H_2)_{\mathbb{Q}}.
\end{align}
which agrees with the Jacobson-Morozov filtration associated to $(H_2, \tilde{\sigma}\subset \mathrm{End}_{\mathbb{Q}}(H_2))$. Therefore we have the following:

\begin{prop}\label{inducedlmhsclassical}
$(W(\sigma)[-1]|_{H_2}, \tilde{F}_0^{\bullet},  \sigma|_{H_2})$ gives a limiting mixed Hodge structure on $((H_2)_{\mathbb{Z}}, \tilde{Q})$ of Hodge-Tate type.
\end{prop}

From now we denote
\begin{equation}
    (H_2)_{\mathbb{C}}=\oplus_{p,q}H_2^{p,q}
\end{equation}
and
\begin{equation}
    \tilde{\mathfrak{g}}_{\mathbb{C}}=\oplus_{p,q}\tilde{\mathfrak{g}}^{p,q}.
\end{equation}
as the splittings for LMHS $(W(\sigma)[-1]|_{H_2}, \tilde{F}_0^{\bullet},  \sigma|_{H_2})$ and its adjoint.

Denote $\tilde{\sigma}, \tilde{\tau}$ as the image of $\sigma, \tau$ under $\mathfrak{g}^{-1,-1}\cong \tilde{\mathfrak{g}}^{-1,-1}$. This isomorphism identifies the topology of $\sigma\cap \mathrm{Ad}_{\gamma_i}\tau$ and $\tilde{\sigma}\cap \mathrm{Ad}_{\tilde{\gamma}_i^l}\tilde{\tau}$. Therefore, the condition of Proposition \ref{uniformizationofbasepoint} implies $\tilde{\sigma}, \tilde{\tau}\in \tilde{\Sigma}$ are two monodromy nilpotent cones, and $\{\tilde{\gamma}_i^l\}\subset  P_{W(\tilde{\sigma})}\leq \mathrm{Sp}(H_2, \mathbb{Z})$ such that there exists 
\begin{equation}
    0\neq \tilde{N_i}\in \tilde{\mathfrak{g}}^{-1,-1}\cong \mathfrak{g}^{-1,-1},
\end{equation}
and the triple: 
\begin{center}\label{inducedbuster}
    $(\tilde{\sigma}, \tilde{F}_0^{\bullet}), (\mathrm{Ad}_{\tilde{\gamma}_i^l}\tilde{\tau}, \tilde{F}_0^{\bullet}), (\tilde{N_i}, \tilde{F}_0^{\bullet})$
\end{center}
are all nilpotent orbits rendering Hodge-Tate type LMHS. This completes the transition, and hence Proposition \ref{finiteintersectionclassical} for $\sigma, \tau\in \Sigma_{\mathrm{HT}}$ implies Proposition \ref{finiteintersectionclassical} for any $\sigma, \tau\in \Sigma$.

\subsection{Concluding remarks}

To complete the proof of Theorem \ref{mainthmclassical}, we still need to fill the gap between the local statement (i.e., Proposition \ref{finiteintersectionclassical}) and the global one. We make the following claim:
\begin{claim}\label{basicclaim}
    Proposition \ref{finiteintersectionclassical} implies Theorem \ref{mainthmclassical}.
\end{claim}

The proof of this claim will be left to Section 7. This completes the proof of Theorem \ref{mainthmclassical}.

%% file: Sec5_CY3.tex
\section{Weak fan of restricted types: The weight $3$ Calabi-Yau case}

In this section we move to non-classical cases. Suppose $\varphi: S\rightarrow \Gamma\backslash D$ is a period map of type $(1,h,h,1)$ where $\Gamma\leq \mathrm{Sp}(2h+2, \mathbb{Z})$ is assumed to be neat\footnote{Any arithmetic subgroup contains a neat subgroup of finite index.}.
 
 As before, let $\Sigma$ be the collection of all local monodromy nilpotent cones arising from $\varphi$. We also assume any nilpotent orbit coming from $\varphi$ via Schmid's nilpotent orbit theorem is of type $\Phi$, where
 \begin{equation}
     \Phi:=\{\hbox{Pure LMHS}, \ \hbox{Type I LMHS}, \ \hbox{Type IV LMHS} \}.
 \end{equation}
Here the LMHS types are defined by \cite[Example 5.8]{KPR19}. This is equivalent to say any divisorial degeneration arising from $\varphi$ is not of type II or type III. The family analyzed in \cite{Den22} is such an example, see Section 6 for a summary. In this section we will prove the following theorem.

\begin{theorem}\label{maintheoremCY3}
For the period map $\varphi$ with the above properties, there exists a polyhedral subdivision of $\Sigma$, denoted as $\hat{\Sigma}$, such that $\hat{\Sigma}$ is a type $\Phi$ weak fan strongly compatible with $\Gamma$ and finitely generated under the action by $\mathrm{Ad}_{\Gamma}$. In other words, for any $\sigma, \tau\in \hat{\Sigma}$, $\tilde{B}(\sigma, \Phi)\cap \tilde{B}(\tau, \Phi)\neq \emptyset$ and $\sigma\cap \tau\neq \emptyset$ imply $\sigma=\tau$.
\end{theorem}

This is an analog of Theorem \ref{mainthmclassical}, and we should proceed in similar ways. Assume $\sigma, \tau\in \Sigma$ are two monodromy nilpotent cones. Suppose there exists a countable sequence $\{\gamma_i\}_{i\geq 0}\in \Gamma$ such that there exists $0\neq N_i\in \sigma\cap \mathrm{Ad}_{\gamma_i}\tau$, and there exists $F_i^{\bullet}\in \check{D}$, such that 
\begin{equation}\label{busterCY3}
    (\sigma, F_i^{\bullet}), (\mathrm{Ad}_{\gamma_i}\tau, F_i^{\bullet}), (N_i, F_i^{\bullet})
\end{equation}
are all nilpotent orbits of type $\Phi$. Again we assume $\gamma_0=\mathrm{Id}$. We should check the two non-trivial LMHS types in $\Phi$ separately in order to prove the following analog of Proposition \ref{finiteintersectionclassical}:

\begin{prop}\label{finiteintersectionCY3}
For any triple \eqref{busterCY3} there exists a finite polyhedral decomposition $\Sigma(\sigma)$ of $\sigma$ such that $\sigma \cap \mathrm{Ad}_{\gamma_i}\tau$ is the union of some $\{\sigma_k\}\subset \Sigma(\sigma)$ for any $i$.
\end{prop}

\subsection{The type I case}\label{type1}

The type $\mathrm{I}_a \ (0\leq a\leq h)$ LMHS for the $(1,h,h,1)$ is characterized by Hodge-Deligne diagrams in Figure \ref{typeonelmhs}.\\

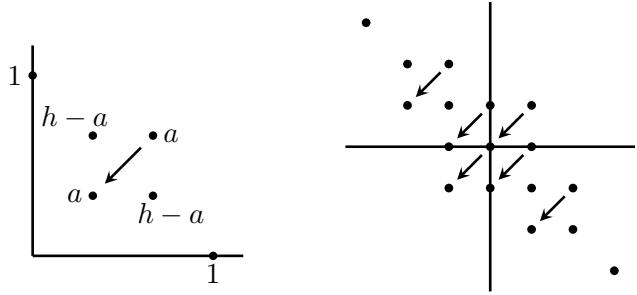
\begin{figure}\label{typeonelmhs}
\centering
    
    \begin{tikzpicture}[scale=0.8]
    \draw[-,line width=1.0pt] (0,0) -- (0,3.5);
    \draw[-,line width=1.0pt] (0,0) -- (3.5,0);
    \fill (0,3) circle (2pt);
    \node at (-0.3,3) {$1$};
    
    \fill (1,1) circle (2pt);
    \node at (0.7,1) {$a$};

     \fill (1,2) circle (2pt);
    \node at (0.7,2.3) {$h-a$};
    
    \fill (3,0) circle (2pt);
    \node at (3,-0.3) {$1$};
    
    \fill (2,1) circle (2pt);
    \node at (2.3,0.7) {$h-a$};
    
    \fill (2,2) circle (2pt);
    \node at (2.3,2) {$a$};
    
    \draw [-stealth](1.8,1.8) -- (1.2,1.2) [line width = 1.0pt];

    \end{tikzpicture}
    $
    \mspace{50mu}
    $
    \begin{tikzpicture}[scale=0.55]
    \draw[-,line width=1.0pt] (-3.5,0) -- (3.5,0);
    \draw[-,line width=1.0pt] (0,-3.5) -- (0,3.5);
    
    \fill (0,0) circle (3pt);
    \fill (1,0) circle (3pt);
    \fill (1,1) circle (3pt);
    \fill (1,-1) circle (3pt);
    \fill (0,1) circle (3pt);
    \fill (0,-1) circle (3pt);
    \fill (-1,0) circle (3pt);
    \fill (-1,1) circle (3pt);
    \fill (-1,-1) circle (3pt);
    
    \fill (2,-1) circle (3pt);
    \fill (1,-2) circle (3pt);
    \fill (2,-2) circle (3pt);
    \fill (3,-3) circle (3pt);
    
    \fill (-2,1) circle (3pt);
    \fill (-2,2) circle (3pt);
    \fill (-1,2) circle (3pt);
    \fill (-3,3) circle (3pt);
    
    \draw [-stealth](-1.2,1.8) -- (-1.8, 1.2) [line width = 1.0pt];
    \draw [-stealth](1.8,-1.2) -- (1.2, -1.8) [line width = 1.0pt];
    
    \draw [-stealth](-0.2,0.8) -- (-0.8,0.2) [line width = 1.0pt];
    \draw [-stealth](0.8,-0.2) -- (0.2,-0.8) [line width = 1.0pt];
    
    \draw [-stealth](0.8,0.8) -- (0.2,0.2) [line width = 1.0pt];
    \draw [-stealth](-0.2,-0.2) -- (-0.8,-0.8) [line width = 1.0pt];
   
    \end{tikzpicture}
    \caption{Type $\mathrm{I}_a$ LMHS for $(1,h,h,1)$}
\end{figure}
    
Suppose $(W(N)[-3], F^{\bullet}, N)$ where $N\in \mathfrak{g}_{\mathbb{Q}}$, $F^{\bullet}\in \check{D}$ gives a type I LMHS. Applying \eqref{tangentspaceofbn} and \eqref{structureofbn} on type I LMHS, we have 
\begin{equation}
(P_N)_{\mathbb{C}}\backslash \mathrm{Stab}_{F^{\bullet}}\cong Z(N)\backslash \mathrm{Stab}_{F^{\bullet}}
\end{equation}
which implies for any monodromy nilpotent cones $\sigma$ and $\tau$, if $W(\sigma)=W(\tau)$ and $\tilde{B}(\sigma, \Phi)\cap \tilde{B}(\tau, \Phi)\neq \emptyset$, then $\tilde{B}(\sigma, \Phi)=\tilde{B}(\tau, \Phi)$. 

Let $\Sigma_{\mathrm{I}}\subset \Sigma$ be the subcollection containing all nilpotent cones which could polarize type I LMHS. Now regarding the triples in \ref{busterCY3}, we consider the case when $\sigma, \tau\in \Sigma_{\mathrm{I}}$ and all of 
\begin{center}
    $(\sigma, F_i^{\bullet}), (\mathrm{Ad}(\gamma_i)\tau, F_i^{\bullet}), (N_i, F_i^{\bullet})$
\end{center}
are nilpotent orbits with type I LMHS. Indeed, we have an analog of Proposition \ref{uniformizationofbasepointHT} for type I degenerations for which the proof is similar:
\begin{prop}\label{uniformizationofbasepointtypeI}
We can choose all $F_i^{\bullet}=F_0^{\bullet}$. In other words,
\begin{center}\label{normalizedbusterone}
    $(\sigma, F_0^{\bullet}), (\mathrm{Ad}_{\gamma_i}\tau, F_0^{\bullet}), (N_i, F_0^{\bullet})$
    \end{center}
    are all nilpotent orbits with type I LMHS.
\end{prop}
Consider the LMHS $(W(\sigma)[-3], F_0^{\bullet}, \sigma)$. The corresponding Deligne splittings are denoted as:
\begin{equation}
    H_{\mathbb{C}}=\oplus_{p,q}H^{p,q}
\end{equation}
and
\begin{equation}
    \mathfrak{g}_{\mathbb{C}}=\oplus_{p,q}\mathfrak{g}^{p,q}.
\end{equation}
with $\sigma$, $\mathrm{Ad}_{\gamma_i}\tau\subset \mathfrak{g}^{-1,-1}$. Consider the splitting:
\begin{equation}
    H_{\mathbb{C}}=H_1\oplus H_2
\end{equation}
where
\begin{align}
    H_1&:=H^{3,0}\oplus H^{2,1}\oplus H^{1,2}\oplus H^{0,3}\\
    \nonumber H_2&:=H^{2,2}\oplus H^{1,1}.
\end{align}
Clearly $H_2$ is defined over $\mathbb{R}$ and $N$ only acts on $H_2$ non-trivially, we denote $\bar{N}\in \mathrm{End}(H_2)$ as an nilpotent endomorphism of $H_2$.

Applying methods in Section \ref{sectionallclassicaldegenerations} in a similar way, let
\begin{equation}\label{intbasis}
\mathcal{B}:=\{e_1,...,e_a,f_0,...,f_{h-a},f^{h-a},...,f^0,e^a,...,e^1\}
\end{equation}
be an integral basis of $(H_{\mathbb{Q}}, Q)$ such that:
\begin{align}
    &W(\sigma)_{-1}\cong \mathrm{span}\{e_1,...,e_a\};\\
   \nonumber &W(\sigma)_0\cong \mathrm{span}\{e_1,...,e_a,f_0,f^0,...,f_{h-a},f^{h-a}\};\\
   \nonumber &W(\sigma)_1\cong H_{\mathbb{Q}}.
\end{align}
and the (alternating) polarization form $Q$ satisfies:
\begin{align}
    &Q(e^i,e_j)=\delta^i_j, 1\leq i,j\leq a\\
   \nonumber &Q(f^i,f_j)=\delta^i_j, 0\leq i,j\leq h-a.
\end{align}
and $Q(\alpha, \beta)=0$ for other cases when $\alpha, \beta\in \mathcal{B}$.

Writing $\tilde{Q}:=Q|_{H_2}$. We also consider the induced Hodge filtration of $F_0^{\bullet}$ on $H_2$ as:
\begin{align}
    &\tilde{F}_0^3:=\{0\},\\
   \nonumber &\tilde{F}_0^2:=\mathrm{img}\{H^{2,2}\hookrightarrow H_2\},\\
   \nonumber &\tilde{F}_0^1=\tilde{F}_0^0:=H_{2, \mathbb{C}}.
\end{align}

Apply the same procedures in Section \ref{sectionallclassicaldegenerations}, we obtain similar statements as Proposition \ref{integralstructure} and Proposition \ref{inducedlmhsclassical} but for the weight $3$ type I degeneration:
\begin{prop}
    There exists a basis of $H_{2,\mathbb{R}}$ under which $\tilde{Q}$ has the standard form \eqref{inducedformforclassical}, and under this basis $(H_2, \tilde{Q})$ can be regarded as an integral symplectic lattice.
\end{prop}

\begin{prop}
     By setting the compact dual $\check{D}_{(a,a)}$ of the period domain $D_{(a,a)}$ parametrizing weight $1$ polarized $\mathbb{Z}$-Hodge structures on $((H_2)_{\mathbb{Z}}, \tilde{Q})$, we can identify $\tilde{F}_0^{\bullet}$ as an element in $\check{D}_{(a,a)}$, and $(W(\sigma)[-3]|_{H_2}, \tilde{F}_0^{\bullet},  \sigma|_{H_2})$ gives a limiting mixed Hodge structure on $H_2$ of Hodge-Tate type. 
\end{prop}
Again we denote
\begin{equation}
    (H_2)_{\mathbb{C}}=\oplus_{p,q}H_2^{p,q}
\end{equation}
and
\begin{equation}
    \tilde{\mathfrak{g}}_{\mathbb{C}}=\oplus_{p,q}\tilde{\mathfrak{g}}^{p,q}.
\end{equation}
as the splittings for LMHS $(W(\sigma)[-3]|_{H_2}, \tilde{F}_0^{\bullet},  \sigma|_{H_2})$ and its adjoint.

As in the classical case, denote $\tilde{\sigma}, \tilde{\tau}$ as the image of $\sigma, \tau$ under $\mathfrak{g}^{-1,-1}\cong \tilde{\mathfrak{g}}^{-1,-1}$. This isomorphism identifies the topology of $\sigma\cap \mathrm{Ad}_{\gamma_i}\tau$ and $\tilde{\sigma}\cap \mathrm{Ad}_{\tilde{\gamma}_i^l}\tilde{\tau}$. Therefore, the condition of Proposition \ref{uniformizationofbasepointtypeI} implies $\tilde{\sigma}, \tilde{\tau}$ are two monodromy nilpotent cones, and $\{\tilde{\gamma}_i\}\subset  P_{W(\tilde{\sigma})}\cap \mathrm{Sp}(H_2, \mathbb{Z})$ such that there exists 
\begin{equation}
    0\neq \tilde{N_i}\in \tilde{\mathfrak{g}}^{-1,-1}\cong \mathfrak{g}^{-1,-1},
\end{equation}
and the triple: 
\begin{center}\label{inducedbusterone}
    $(\tilde{\sigma}, \tilde{F}_0^{\bullet}), (\mathrm{Ad}_{\tilde{\gamma}_i}\tilde{\tau}, \tilde{F}_0^{\bullet}), (\tilde{N_i}, \tilde{F}_0^{\bullet})$
\end{center}
are all nilpotent orbits with Hodge-Tate type LMHS. This completes the transition from the type I degeneration of $(1,h,h,1)$ to the Hodge-Tate degeneration of weight $1$ case, hence Proposition \ref{finiteintersectionclassical} implies Proposition \ref{finiteintersectionCY3} when $\sigma, \tau\in \Sigma_{\mathrm{I}}$.

\subsection{The type IV case}

Now we consider the type IV degeneration of $(1,h,h,1)$. Its Hodge-Deligne diagrams are shown in Figure \ref{typeIVlmhs}. In particular, when $a=h$, it gives the Hodge-Tate degeneration (of weight $3$ Calabi-Yau case).
\begin{figure}\label{typeIVlmhs}
\centering
    
    \begin{tikzpicture}[scale=0.8]
    \draw[-,line width=1.0pt] (0,0) -- (0,3.5);
    \draw[-,line width=1.0pt] (0,0) -- (3.5,0);
    \fill (0,0) circle (2pt);
    \node at (-0.3,-0.3) {$1$};
    
    \fill (1,1) circle (2pt);
    \node at (0.7,1) {$a$};

     \fill (1,2) circle (2pt);
    \node at (0.7,2.3) {$h-a$};
    
    \fill (3,3) circle (2pt);
    \node at (3.3,3.3) {$1$};
    
    \fill (2,1) circle (2pt);
    \node at (2.3,0.7) {$h-a$};
    
    \fill (2,2) circle (2pt);
    \node at (2.3,2) {$a$};
    
    \draw [-stealth](2.8,2.8) -- (2.2,2.2) [line width = 1.0pt];
    \draw [-stealth](1.8,1.8) -- (1.2,1.2) [line width = 1.0pt];
    \draw [-stealth](0.8,0.8) -- (0.2,0.2) [line width = 1.0pt];

    \end{tikzpicture}
    $
    \mspace{50mu}
    $
    \begin{tikzpicture}[scale=0.55]
    \draw[-,line width=1.0pt] (-3.5,0) -- (3.5,0);
    \draw[-,line width=1.0pt] (0,-3.5) -- (0,3.5);
    
    \fill (0,0) circle (3pt);
    \fill (1,0) circle (3pt);
    \fill (1,1) circle (3pt);
    \fill (1,-1) circle (3pt);
    \fill (0,1) circle (3pt);
    \fill (0,-1) circle (3pt);
    \fill (-1,0) circle (3pt);
    \fill (-1,1) circle (3pt);
    \fill (-1,-1) circle (3pt);
    
    \fill (2,1) circle (3pt);
    \fill (1,2) circle (3pt);
    \fill (2,2) circle (3pt);
    \fill (3,3) circle (3pt);
    
    \fill (-2,-1) circle (3pt);
    \fill (-2,-2) circle (3pt);
    \fill (-1,-2) circle (3pt);
    \fill (-3,-3) circle (3pt);
    
    \draw [-stealth](-2.2,-2.2) -- (-2.8, -2.8) [line width = 1.0pt];
    \draw [-stealth](2.8,2.8) -- (2.2, 2.2) [line width = 1.0pt];
    \draw [-stealth](1.8,1.8) -- (1.2, 1.2) [line width = 1.0pt];
    \draw [-stealth](-1.2,-1.2) -- (-1.8,-1.8) [line width = 1.0pt];
    \draw [-stealth](0.8,0.8) -- (0.2,0.2) [line width = 1.0pt];
    \draw [-stealth](-0.2,-0.2) -- (-0.8,-0.8) [line width = 1.0pt];
    
    \draw [-stealth](-0.2,0.8) -- (-0.8,0.2) [line width = 1.0pt];
    \draw [-stealth](0.8,1.8) -- (0.2,1.2) [line width = 1.0pt];
    \draw [-stealth](-1.2,-0.2) -- (-1.8,-0.8) [line width = 1.0pt];
    
    \draw [-stealth](1.8,0.8) -- (1.2,0.2) [line width = 1.0pt];
    \draw [-stealth](0.8,-0.2) -- (0.2,-0.8) [line width = 1.0pt];
    \draw [-stealth](-0.2,-1.2) -- (-0.8,-1.8) [line width = 1.0pt];
    
    \end{tikzpicture}
    \caption{Type $\mathrm{IV}_a$ LMHS for $(1,h,h,1)$}
\end{figure}
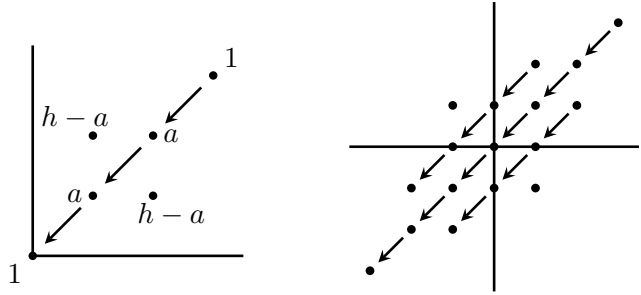

Again the corresponding Deligne splittings are denoted as:
\begin{equation}
    H_{\mathbb{C}}=\oplus_{p,q}H^{p,q}
\end{equation}
and
\begin{equation}
    \mathfrak{g}_{\mathbb{C}}=\oplus_{p,q}\mathfrak{g}^{p,q}.
\end{equation}

Let $\Sigma_{\mathrm{IV}}\subset \Sigma$ be the subcollection containing all nilpotent cones which could polarize type IV LMHS. Notice that $\Sigma_{\mathrm{IV}}$ is in general not closed under taking faces. Now regarding the triples in \eqref{busterCY3}, we consider the case when $\sigma, \tau\in \Sigma_{\mathrm{IV}}$ and all of 
\begin{equation}\label{bustertypeIV}
    (\sigma, F_i^{\bullet}), (\mathrm{Ad}_{\gamma_i}\tau, F_i^{\bullet}), (N_i, F_i^{\bullet})
\end{equation}
are nilpotent orbits with type IV LMHS. In this subsection we prove Proposition \ref{finiteintersectionCY3} when $\sigma, \tau\in \Sigma_{\mathrm{IV}}$.

If $\sigma\subset \mathfrak{g}_{\mathbb{Q}}$ is a monodromy nilpotent cone polarizing type IV LMHS, its Jacobson-Morozov weight filtration $W(\sigma)$ on $H_{\mathbb{Q}}$ can be described as follows. Let $(H_{\mathbb{Z}}, Q)$ be the underlying lattice. Take an integral basis $\mathcal{B}$ under which $Q$ has the same form as \eqref{polarizationformclassical}, where $\mathfrak{B}$ and $W(\sigma)$ are: 
\begin{equation}
\mathcal{B}:=\{\omega_-,e_1,...,e_a,f_1,...,f_{h-a},f^{h-a},...,f^1,e^a,...,e^1,\omega_+\},
\end{equation}
and
\begin{align}
    &W(\sigma)_3\cong H_{\mathbb{Q}}\\
   \nonumber W(\sigma)_2=&W(\sigma)_1\cong \mathrm{span}\{\omega_-,e_1,...,e_a,f_1,...,f_{h-a},f^{h-a},...,f^1,e^a,...,e^1\}\\
   \nonumber &W(\sigma)_0\cong \mathrm{span}\{\omega_-,e_1,...,e_a,f_1,...,f_{h-a},f^{h-a},...,f^1\}\\
   \nonumber &W(\sigma)_{-1}\cong \mathrm{span}\{\omega_-,e_1,...,e_a\}\\
   \nonumber &W(\sigma)_{-2}=W(\sigma)_{-3}\cong \mathrm{span}\{\omega_-\}.
\end{align}

Suppose for any two type IV monodromy nilpotent cones $\sigma, \tau$, there exists a (countable) sequence $\{\gamma_i\}_{i\geq 0}\subset \Gamma$ such that there exists $0\neq N_i\in \sigma\cap \mathrm{Ad}_{\gamma_i}\tau$ and $\{F_i^{\bullet}\}\subset \check{D}$, we have
\begin{center}\label{bustertype4}
    $(\sigma, F_i^{\bullet}), (\mathrm{Ad}_{\gamma_i}\tau, F_i^{\bullet}), (N_i, F_i^{\bullet})$
\end{center}
are all nilpotent orbits.

Denote 
\begin{equation}\label{quotientvectorspaceH}
    \tilde{H}:=W(\sigma)_1/W(\sigma)_{-2},
\end{equation}
that is,
\begin{equation}
    \tilde{H}\cong\mathrm{span}\{e_1,...,e_a,f_1,...,f_{h-a},f^{h-a},...,f^1,e^a,...,e^1\}
\end{equation}
Denote $\tilde{F}_i^{\bullet}$ as the filtration on $\tilde{H}$ induced by $F_i^{\bullet}$. It is straightforward to check if we regard $\check{D}_{(h,h)}$ as the compact dual of $D_{(h,h)}$ which parametrizes $\mathbb{Z}$-polarized Hodge structures on $(\tilde{H}, \tilde{Q}=Q|_{\tilde{H}})$, then $\tilde{F}_i^{\bullet}\in \check{D}_{(h,h)}$\footnote{More precisely, we are identifying $D_{(0,h,h,0)}$ and $D_{(h,h)}$ by forgetting two trivial factors and twisting by $\mathbb{Z}[1]$.}.

Denote 
\begin{equation}\label{projectionoftypeIVcones}
\tilde{\sigma}:= \mathrm{img}\{\sigma\subset \mathfrak{g}_{\mathbb{Q}}\cap \mathfrak{g}^{-1,-1}\rightarrow \mathrm{End}(\tilde{H}_{\mathbb{Q}})\}.
\end{equation}
It is well defined since $\sigma$ annihilates $W(\sigma)_{-2}$. We also denote $W(\tilde{\sigma})$ on $\tilde{H}_{\mathbb{Q}}$ by:
\begin{align}
    &W(\tilde{\sigma})_1\cong \mathrm{span}\{e_1,...,e_a,f_1,...,f_{h-a},f^{h-a},...,f^1,e^a,...,e^1\}\\
    &\nonumber W(\tilde{\sigma})_0\cong \mathrm{span}\{e_1,...,e_a,f_1,...,f_{h-a},f^{h-a},...,f^1\}\\
   &\nonumber W(\tilde{\sigma})_{-1}\cong \mathrm{span}\{e_1,...,e_a\}.
\end{align}
It is now clear that:
\begin{prop}\label{typeIVquotient}
    On $(\tilde{H}, \tilde{Q}=Q|_{\tilde{H}})$, $(W(\tilde{\sigma})[-1], \tilde{F}_i^{\bullet}, \tilde{\sigma})$ gives a limiting mixed Hodge structure (of classical weight $1$ type).
\end{prop}
We denote this LMHS and its adjoint as:
\begin{equation}\label{hodgedecompositiontypeIVquotient}
    \tilde{H}_{\mathbb{C}}=\oplus_{p,q}\tilde{H}^{p,q}
\end{equation}
and
\begin{equation}
    \tilde{\mathfrak{g}}_{\mathbb{C}}=\oplus_{p,q}\tilde{\mathfrak{g}}^{p,q}.
\end{equation}

Another observation is 
\begin{equation}
   \{\gamma_i\}\subset \mathrm{Sp}(2h+2,Q;\mathbb{Z})\cap P_{W(\sigma)}
\end{equation}
 has induced actions
 \begin{equation}
   \{\tilde{\gamma}_i\}\subset \mathrm{Sp}(2h,\tilde{Q};\mathbb{Z})\cap P_{W(\tilde{\sigma})}
\end{equation}
by restricting on $W(\sigma)_2$ and taking quotient by $W(\sigma)_{-2}$. 

\medskip
Let $\tilde{\sigma}, \tilde{\tau}$ be the projection of $\sigma, \tau$ on $\tilde{\mathfrak{g}}^{-1,-1}$ under \eqref{projectionoftypeIVcones}. Notice that unlike the previous cases, $\sigma\cap \mathrm{Ad}_{\gamma_i}\tau$ and $\tilde{\sigma}\cap \mathrm{Ad}_{\tilde{\gamma}_i}\tilde{\tau}$ may have different topologies. Since for any $\sigma, \tau\in \Sigma_{\mathrm{IV}}$ , any $\{\gamma_i\}\subset \Gamma$, and any triples of nilpotent orbits
\begin{center}
    $(\sigma, F_i^{\bullet}), (\mathrm{Ad}_{\gamma_i}\tau, F_i^{\bullet}), (N_i, F_i^{\bullet})$,
\end{center}
with $N_i\in \sigma\cap \mathrm{Ad}_{\gamma_i}\tau$, we have the induced triples
\begin{center}
    $(\tilde{\sigma}, \tilde{F}_i^{\bullet}), (\mathrm{Ad}_{\tilde{\gamma}_i}\tilde{\tau}, \tilde{F}_i^{\bullet}), (\tilde{N}_i, \tilde{F}_i^{\bullet})$.
\end{center}
which are all nilpotent orbits producing classical weight $1$ LMHS on $(\tilde{H}, \tilde{Q})$. Therefore, by applying Proposition \ref{finiteintersectionclassical} to $(\tilde{\sigma}, \tilde{\tau})$ we get a subdivision $\Sigma(\tilde{\sigma})$ of $\tilde{\sigma}$. 

Let $\Sigma(\sigma)$ be the subdivision of $\sigma$ obtained by intersecting $|\sigma|$ with the preimage of $\Sigma(\tilde{\sigma})$ under the projection map \eqref{projectionoftypeIVcones}. The following lemma indicates $\Sigma(\sigma)$ verifies Proposition \ref{finiteintersectionCY3} when $\sigma, \tau\in \Sigma_{\mathrm{IV}}$.

\begin{lemma}\label{typeIVtricklemma}
    Suppose $\sigma, \tau\in \Sigma_{\mathrm{IV}}$ such that $\tilde{B}(\sigma, \Phi)\cap \tilde{B}(\tau, \Phi)\neq \emptyset$, $\sigma\cap \tau\neq \emptyset$ and $\tilde{\sigma}=\tilde{\tau}$ in $\tilde{\Sigma}_{\mathrm{IV}}$, then $\sigma=\tau$.
\end{lemma}
\begin{proof}
    For $F^{\bullet}\in \tilde{B}(\sigma, \Phi)\cap \tilde{B}(\tau, \Phi)$ let $H_{\mathbb{C}}=\oplus H^{p,q}$ be the Hodge decomposition of LMHS $(W(\sigma)[-3], F^{\bullet}, \sigma)$. Suppose the lemma is false, there exist $N_1\in \sigma\backslash \tau$ and $N_2\in \tau$ such that $\tilde{N}_1=\tilde{N}_2$ in $\tilde{\sigma}=\tilde{\tau}$, this implies $N_1-N_2\neq 0$ annihilates $\tilde{H}$. On the other hand, take $0\neq N_3\in \sigma\cap\tau$, we know $[N_1-N_2, N_3]=0$ and $(N_3, F^{\bullet})$ is a nilpotent orbit with type IV LMHS, this leads to a contradiction since for $0\neq \mathbf{v}\in H^{3,3}$, $(N_1-N_2)N_3\mathbf{v}=0$ but $N_3(N_1-N_2)\mathbf{v}\neq 0$.
\end{proof}

\subsection{Concluding remarks}

Like Claim \ref{basicclaim} for the classical case, in this section we should make the following claim.

\begin{claim}\label{basicclaim2}
  Proposition \ref{finiteintersectionCY3} implies Theorem \ref{maintheoremCY3}.
\end{claim}
The proof is a bit different from the proof of Claim \ref{basicclaim} and we also leave it to section 7.

A comment on the necessity of using weak fan of restricted type is elements in $\Sigma_{\mathrm{I}}$ could also polarize type II LMHS, as Figure \ref{lmhstypefigure} shows. We are not able to show the finiteness results like Theorem \ref{maintheoremCY3} if $\Phi$ includes type II LMHS at this time.

%% file: Sec6_Example.tex
\section{Subdivision, logarithmic modification and example}

\subsection{Overview}
Given a period map $\varphi: S\rightarrow \Gamma\backslash D$ which are of types described above, Theorem \ref{mainthmclassical} and \ref{maintheoremCY3} are combinatorial statements that we can modify the collection of local monodromy nilpotent cones $\Sigma$ in a proper way to make it fit Kato-Nakayama-Usui's theory. To fully establish the main result regarding period maps themselves (see Section \ref{introsection}), we shall explain what do Theorem \ref{mainthmclassical} and \ref{maintheoremCY3} mean in the Hodge-theoretical side.

Indeed, in Theorem \ref{mainthmclassical} and \ref{maintheoremCY3}, since the subdivision on $\Sigma$ is finitely generated under $\mathrm{Ad}_{\Gamma}$-action, it can be generated by a finite subdivision on the set of local monodromy nilpotent cones coming from the period map $\varphi: S\rightarrow \Gamma\backslash D$. This operation is called "logarithmic modification" in Kato-Usui's theory.

The general theory of logarithmic modification is introduced in \cite[Sec. 3.6]{KU08}, in which the base space $S$ is assumed to be an object in an abstract category $\mathcal{B}(\mathrm{log})$. In our settings where $S$ is quasi-projective with normal crossing divisors, logarithmic modification is exactly local toric blow-ups\footnote{See \cite[Chap. 11]{CLS11} for a comprehensive introduction on toric resolutions and polyhedral subdivisions.}. Indeed, a finite (rational) subdivision on local monodromy cones corresponds to a finite blow-up sequence:
\begin{equation}\label{finiteblowupsequence}
    \pi: \overline{S_M}\rightarrow \overline{S_{M-1}}\rightarrow ... \rightarrow \overline{S_1}\rightarrow \overline{S_0}=\overline{S},
\end{equation}
where each arrow is a (weighted) blow-up along some subvariety. As a consequence of finiteness, $S_M:=\pi^{-1}(S)$ is also quasi-projective, and the original period map can be lifted to:
\begin{equation}\label{liftedperiodmap}
\begin{tikzcd}
S_M \arrow[d] \arrow[dr, "\varphi_M"] & \\
S \arrow[r, "\varphi"] & \Gamma \backslash D.
\end{tikzcd}
\end{equation}
Combine this with Theorems \ref{mainthmclassical} and \ref{maintheoremCY3}, the main result of this paper summarized in Section \ref{introsection} can be formally stated as follows:

\begin{theorem}\label{formalmainthm}
    Suppose $\varphi: S\rightarrow \Gamma\backslash D$ is a period map such that:
    \begin{description}
        \item[(i)] $S$ is quasi-projective with normal crossing boundary divisors;
        \item[(ii)] The period map is either of classical weight $1$ type, or weight $3$ Calabi-Yau type which acquires only type I or type IV degenerations,
    \end{description} 
 then there exists a quasi-projective variety $S_M$ and a finite blow-up sequence \eqref{finiteblowupsequence}, such that there exists a $\Gamma$-strongly compatible weak fan (of restricted type) $\Sigma$, and the lifted period map \eqref{liftedperiodmap} admits an Kato-Usui type extension:
 \begin{equation}
     \overline{\varphi_M}: \overline{S_M}\rightarrow \Gamma\backslash D_{\Sigma, \Phi}
 \end{equation}
in which $\Phi$ is a set of LMHS types, $\Gamma\backslash D_{\Sigma, \Phi}$ is a locally analytically constructible space and $\overline{\varphi_M}$ is a morphism of locally analytically constructible spaces.
\end{theorem}

\subsection{An example on the $2$-dimensional case}
In this subsection we work out a $2$-dimensional example to show how logarithmic modification can make changes on the period map. 

Consider the local period map:
\begin{equation}
    \varphi: (\Delta^{*})^2\rightarrow \Gamma\backslash D,
\end{equation}
where under a fixed coordinate system $(x,y)\in \Delta^2$, the (unipotent) monodromy operators around coordinate divisors $D_x:=\{x=0\}, D_y:=\{y=0\}$ are $T_x, T_y$, and $\Gamma=\langle T_x, T_y\rangle$ is abelian.

Consider the blow-up of $(\Delta^{*})^2$ at the point $(0,0)$, the blow-up map is read as:
\begin{equation}
    \pi: \widehat{\Delta^2}\rightarrow {\Delta^2}.
\end{equation}
Denote the exception divisor $E$. The same method in \cite[Sec. 6.1]{Den22} shows for the lifted period map:
\begin{equation}
    \hat{\varphi}: \widehat{\Delta^2}\rightarrow \Gamma\backslash D,
\end{equation}
the monodromy operator around $E$ is $T_xT_y$. Passing to the monodromy logarithms, we have $N_E=N_x+N_y$, which implies lifting the period map is derived from subdividing the monodromy nilpotent cone $\langle N_x,N_y\rangle$ to $\langle N_x, N_x+N_y\rangle\cup \langle N_x+N_y, N_y\rangle$. 

Under the coordinates $(u:=\frac{\log(x)}{2\pi i},v:=\frac{\log(y)}{2\pi i})$, the period map $\varphi$ can be written as:
\begin{equation}\label{localperiodmap2dim}
    \varphi(u,v)=\exp(uN_x+vN_y)\psi(x,y),
\end{equation}
where $\psi(x,y): \Delta^2\rightarrow \check{D}$ is holomorphic. Switch to $\hat{\varphi}$, around the point $o_y:=E\cap D_y$, the local coordinate expression of $\pi$ is:
\begin{equation}
    (x,t)\rightarrow (x,xt,[1:t])\in \widehat{\Delta^2}\xrightarrow[]{\pi}(x,xt)\in \Delta^2.
\end{equation}
This implies around $o_y$, the local lift of the period map \eqref{localperiodmap2dim} is lifted to:
\begin{equation}
    \hat{\varphi}(x,t)=\varphi(x,xt)=\exp(\frac{\log(x)}{2\pi i}(N_x+N_y)+\frac{\log(t)}{2\pi i}N_y)\psi(x,xt).
\end{equation}
Notice that on $E:\{x=0\}$, we have:
\begin{equation}
    \exp(-\frac{\log(x)}{2\pi i}(N_x+N_y)-\frac{\log(t)}{2\pi i}N_y)\hat{\varphi}(0,t)=\psi(0,0)
\end{equation}
is constant. By Schmid's nilpotent orbit theorem, locally for any point $x\in E\backslash (D_x\cup D_y)$, the nilpotent orbit associated to $x$ is $(N_x+N_y, \psi(0,0))$. This implies if we construct the Kato-Usui extension of $\hat{\varphi}$ as Theorem \ref{formalmainthm}, the extended map $\overline{\hat{\varphi}}$ is locally constant on exceptional divisors.
\begin{remark}
    Another observation is since in \eqref{localperiodmap2dim}, $\psi(0,0)$ is well-defined only up to the action of $\mathbb{C}\langle N_x, N_y\rangle$, the value $(N_x+N_y, \psi(0,0))$ is well-defined only after fixing a local coordinate system. In general, the Kato-Usui type extended map in Theorem \ref{formalmainthm} is well defined only after a coordinate system on $S$ is fixed.
\end{remark}

\subsection{Summary of the example in \cite{Den22}}
In this subsection we recall the construction in \cite{Den22}. The period map used here is the one introduced by \cite{HT14} and \cite{HT18}, which is a period map of type $(1,2,2,1)$ as follows: 
\begin{equation}
    \hat{\varphi}: \mathbb{P}^2\backslash \mathrm{Dis}\rightarrow \Gamma\backslash D.
\end{equation}
The discriminant locus is given by:
\begin{equation}
    \mathrm{Dis}:= D_0\cup D_1\cup D_2\cup C
\end{equation}
where $D_i$ are coordinate divisors in $\mathbb{P}^2$, and $C$ is an irreducible quintic with $6$ self-intersection nodes and $1$ tangent point of order $5$ with each coordinate divisor. The picture of this base can be found at \cite[Fig. 6.1]{HT14}. The precise monodromy group $\Gamma$ is not known, but from \cite[Sec. 7]{Den22} we know the algebraic monodromy group $\overline{\Gamma}^{\mathbb{Q}}$ is $\mathrm{Sp}(6,\mathbb{Q})$.

By blowing-up each tangent point for $5$ times, we get a lifted period map:
\begin{equation}
    \varphi: \widetilde{\mathbb{P}^2\backslash \mathrm{Dis}}\rightarrow \Gamma\backslash D.
\end{equation}
where the base $\widetilde{\mathbb{P}^2\backslash \mathrm{Dis}}=:S$ is quasi-projective with normal-crossing boundary divisors. 

We mark all local monodromy nilpotent cones as follows:
\begin{description}\label{alllocalmonodromynilporbits}
    \item[(i)] Let $\sigma_x, \sigma_y, \sigma_z$ be the ($2$-dimensional) local monodromy nilpotent cones around the intersection of each pair of coordinate divisors;
    \item[(ii)] Let $\sigma_0, \sigma_1, \sigma_2$ be the local monodromy nilpotent cones obtained by blowing-up each of the fifth tangent point. As \cite[Sec. 6.1]{Den22} shows, the blow-up process annihilates the order-$5$ semisimple part of the monodromy operators around $C$ and leave the unipotent part invariant;
    \item[(iii)] Let $\tau_j, \ j=1,...,6$ be the local monodromy nilpotent cones obtained from all self-intersection nodes of $C$.
\end{description}
By computational results in \cite{HT14} and \cite{HT18}, all $2$-dimensional local monodromy nilpotent orbits and their Hodge degeneration types are listed as follows (Again we use the type defined by \cite[Example 5.8]{KPR19}).
\begin{description}\label{typeofalllocalmonodromy}
    \item[(i)] $\sigma_x, \sigma_y, \sigma_z$ are all of the type $\langle \mathrm{IV}_2|\mathrm{IV}_2|\mathrm{IV}_2\rangle$;
    \item[(ii)] $\sigma_0, \sigma_1, \sigma_2$ are all of the type
    $\langle\mathrm{IV}_2|\mathrm{IV}_2|\mathrm{I}_1\rangle$;
    \item[(iii)] $\tau_j, \ 1\leq j\leq 6$ are all of the type
    $\langle\mathrm{I}_1|\mathrm{I}_2|\mathrm{I}_1\rangle$.
\end{description}
Where for a $2$-dimensional cone $\sigma=\mathbb{Q}_{\geq 0}\langle N_1,N_2\rangle$ and its associated nilpotent orbit $(\sigma, F^{\bullet})$, type $\langle A|B|C\rangle$ means the type of LMHS given by $(N_1, F^{\bullet})$, $(\sigma, F^{\bullet})$ and $(N_2, F^{\bullet})$ respectively.

Let $\mathcal{S}$ be the set of all monodromy nilpotent cones listed above and all of their proper faces. Clearly $\mathcal{S}$ is a finite set. By our definition, the collection $\Sigma$ with $|\Sigma|\subset \mathfrak{g}_{\mathbb{Q}}$ is the union of $\mathrm{Ad}_{\Gamma}$-orbits of $\mathcal{S}$. One of the main results in \cite{Den22} is showing that after a finite base change which leads to changing $\Gamma$ to a neat subgroup of finite index $\Gamma_0$ as well as a replacing $\Sigma$ by one of its finitely $\mathrm{Ad}_{\Gamma}$-generated subdivisions $\Sigma_0$,  $\Sigma_0$ is a $\Gamma_0$-strongly compatible fan which provides a Kato-Usui style extension. Clearly, Theorem \ref{formalmainthm} provides an alternative proof on the existence of a Kato-Usui type extension.

%% file: Sec7_weakfan.tex
\section{Combinatorics of weak fan}

The purpose of this section is to investigate the collection $\Sigma$ of nilpotent cones generated by local monodromy nilpotent cones coming from a period map $\varphi: S\rightarrow D$. As a conseuqnce, we prove Claim \ref{basicclaim} and Claim \ref{basicclaim2}.

We continue using definitions and notations in Section 4-5. Let $\Sigma$ be a collection of nilpotent cones with $\mathrm{Ad}_{\Gamma}$-generating set:
\begin{equation}
    \mathcal{S}:=\{\sigma_0,\sigma_1,...,\sigma_K\},
\end{equation}
here we can just take $\sigma_i$ to be all local monodromy nilpotent cones coming from $\varphi$ (with a chosen order).

For any ordered pair 
\begin{equation}\label{orderedpairofcones}
    (\sigma_i, \sigma_j)\in \mathcal{S}\times \mathcal{S}, \ 0\leq i,j\leq K, 
\end{equation}
let $\{\gamma^{ij}_l\}_{l\geq 0}\subset \Gamma$ be the countable sequence such that 
\begin{equation}
    \sigma_i\cap \mathrm{Ad}_{\gamma^{ij}_k}\sigma_j \neq \emptyset, \ \tilde{B}(\sigma_i)\cap \tilde{B}(\mathrm{Ad}_{\gamma^{ij}_k}\sigma_j)\neq \emptyset.
\end{equation}
If the sequence is indeed finite (or even empty), we make it an infinite one by adding infinitely many $\mathrm{Id}$. By Proposition \ref{gammafiniteclass}, the image of $\{\gamma^{ij}_l\}_{l\geq 0}$ under the projection
\begin{equation}
     G_{\mathbb{C}}\rightarrow Z(\sigma_i)\backslash G_{\mathbb{C}}/Z(\sigma_j) 
\end{equation}
is finite. Let $L>0$ big enough and for any $(i,j)$, we choose a set 
\begin{equation}
    \{\gamma^{ij}_0,...,\gamma^{ij}_{L}\}\subset \Gamma, \ 0\leq i,j\leq K
\end{equation}
covers all of these classes. By this procedure, for each ordered pair $(i,j)$ and the corresponding ordered pair in \eqref{orderedpairofcones}, we obtain a set 
of the following $(K+1)^2(L+1)$ cones:
\begin{equation}\label{allpossiblebusters}
    \{\sigma_{i}\}\cup \{\mathrm{Ad}_{\gamma^{ij}_l}\sigma_j\}, \ 0\leq i,j\leq K, \ 0\leq l\leq L.
\end{equation}
Next we prove a lemma regarding the combinatorics of polyhedral cones.
\begin{lemma}\label{simplicialfanconstruction}
    For a finite-dimensional $\mathbb{Q}$-vector space $V$ and a finite collection of sharp polyhedral cones $\{\sigma_i\}_{0\leq i\leq N}$ in $V$ which is closed under taking faces, there exists a finite subdivision on $\cup_{0\leq i\leq N}|\sigma_i|$ such that the resulting new collection of cones $\{\tau_j\}_{0\leq j\leq M}$ satisfy:
    \begin{description}
         \item[(i)] $\cup_{0\leq i\leq N}|\sigma_i|=\cup_{0\leq j\leq M}|\tau_j|$,
         \item[(ii)] The collection  $\{\tau_j\}_{0\leq j\leq M}$ is a fan.
    \end{description}
\end{lemma}
\begin{proof}
    Fix a coordinate system $(x_1,..,x_n)$ on $V$. For each cone $\sigma_i$, there exists finitely many linear forms 
    \begin{equation}
        l^i_j:=\sum_{\lambda=1}^{n_i} a^{ij}_{\lambda}x_{\lambda}, \ 0\leq j\leq n_i,
    \end{equation}
    each of which is defined over $\mathbb{Q}$, such that
    \begin{equation}
        |\sigma_i|=\cap_{0\leq j\leq n_i}l^i_j\circ 0
    \end{equation}
    where each $\circ$ should be replaced by one of $>, \ <, \ =$. We shall call this a (rational) hyperplane representation of $\sigma_i$. For each $\sigma_i$ we fix a hyperplane representation. Running over all of $0\leq i\leq N$, these linear forms give finitely many hyperplanes in $V$:
    \begin{equation}
        \{l^i_j=0 \ | \ 0\leq i\leq N, \ 0\leq j\leq n_i\}.
    \end{equation}
    Which seperate the whole space $V$ into finitely many relatively open chambers $\{\tau_j\ | \ j \in \mathcal{J}\}$ (with different dimensions). A direct observation is the support of $\{\tau_j\ | \ j \ \in \mathcal{J}\}$ is $V$, and $\{\tau_j\ | \ j \in \mathcal{J}\}$ gives a refinement of $\{\sigma_i\}$ in the sense that each $\sigma_i$ is the disjoint union of finitely many elements in $\{\tau_j\ | \ j \in \mathcal{J}\}$.
    
    We take the chambers $\{\tau_j \ | \ 1\leq j\leq M\}$ such that the condition (i) in the lemma is satisfied. The condition (ii) is immediately justified by the following observation: Each $\tau_j$ is a relatively open convex polyhedral cone, and $\tau_j\cap \tau_k$ is either $\emptyset$ or $\tau_j=\tau_k$, in which the latter case happens if and only if $\tau_j$ and $\tau_k$ have the same hyperplane representation.
\end{proof}
To apply the lemma, we shall denote 
\begin{equation}
   \mathcal{T}:=\{\tau_j \ | \ 1\leq j\leq M\}
\end{equation}
as the simplicial cone complex constructed from the set of cones in \eqref{allpossiblebusters}, and $\hat{\Sigma}$ be the union of $\mathrm{Ad}_{\Gamma}$-orbits of $\mathcal{T}$. Clearly $|\hat{\Sigma}|=|\Sigma|$ and $\hat{\Sigma}$ is a refinement of $\Sigma$, and the subdivision is finitely generated up to $\mathrm{Ad}_{\Gamma}$-action.

Therefore, to prove Claim \ref{basicclaim} and Claim \ref{basicclaim2}, it suffices to prove $\hat{\Sigma}$ is a $\Gamma$-strongly compatible weak fan in each case. 

\subsection{Proof of Claim \ref{basicclaim}}

A critical observation is the notions of fan and weak fan are equivalent in the classical cases, as a consequence of the following Lemma \ref{opensubconesamebn} and Proposition \ref{fanweakfanequivalent}.

\begin{lemma}\label{opensubconesamebn}
    Suppose $\hat{\sigma}\subset \sigma$ are nilpotent cones and of same dimension. If $\tilde{B}(\sigma)\neq \emptyset$, then $\tilde{B}(\sigma)=\tilde{B}(\hat{\sigma})$.
\end{lemma}
\begin{proof}
    As long as there exists $F^{\bullet}\in \tilde{B}(\sigma)\subset \tilde{B}(\hat{\sigma})$, by \eqref{structureofbn} we have $\tilde{B}(\sigma)= \tilde{B}(\hat{\sigma})$ as $\sigma_{\mathbb{C}}=\hat{\sigma}_{\mathbb{C}}$ by the assumption, where the superscript $\circ$ means the connected components containing the $Z(\sigma)$-orbit of $F^{\bullet}$. However, in the classical weight $1$ case, $Z(\sigma)$ acts transitively on $\tilde{B}(\sigma)$, hence the lemma follows.
\end{proof}

\begin{prop}\label{fanweakfanequivalent}
    Suppose $\sigma, \tau$ are two nilpotent cones polarize LMHS of classical weight $1$ type. If $\sigma\cap \tau\neq \emptyset$, then $\tilde{B}(\sigma)=\tilde{B}(\tau)\neq \emptyset$.
\end{prop}
\begin{proof}
    Fix an $F^{\bullet}\in \tilde{B}(\sigma)$, and denote $\mathfrak{g}=\oplus \mathfrak{g}^{p,q}$ as the adjoint Hodge decomposition of $(\sigma, F^{\bullet})$. We also denote
    \begin{equation}
        \mathcal{C}_{F^{\bullet}}:=\{N\in \mathfrak{g}^{-1,-1}\cap \mathfrak{g}_{\mathbb{Q}} \ | \ (N, F^{\bullet}) \ \hbox{is a nilpotent orbit} \ \}.
    \end{equation}
     Since $\sigma\cap \tau\neq \emptyset$, we must have $W(\sigma)=W(\tau)$, therefore $\tau\subset W_{\mathfrak{g}}(\sigma)_{-2}$, in the classical weight $1$ case this means $\tau\subset \mathfrak{g}^{-1,-1}$ and thus $\mathcal{C}_{F^{\bullet}}\cap \tau_{\mathbb{Q}}$ is open in $\tau_{\mathbb{Q}}$. In other words, there exists a $\tau^{'}\subset \tau$ such that $\sigma\cap \tau\subset \tau^{'}\subset \tau$, $\tau^{'}_{\mathbb{Q}}=\tau_{\mathbb{Q}}$ and $(\tau^{'}, F^{\bullet})$ is a nilpotent orbit. Combine Lemma \ref{opensubconesamebn} and $\tilde{B}(\tau)\neq \emptyset$ we know $(\tau, F^{\bullet})$ is a nilpotent orbit, hence $\tilde{B}(\sigma)\subset \tilde{B}(\tau)$. Similarly we have $\tilde{B}(\sigma)\supset \tilde{B}(\tau)$ and thus $\tilde{B}(\sigma)= \tilde{B}(\tau)$.
\end{proof}

Suppose for some $\hat{\sigma}, \hat{\tau}\in \hat{\Sigma}$ and $\gamma\in \Gamma$, we have $\hat{\sigma} \cap \hat{\tau}\neq \emptyset$. After conjugating by $\Gamma$ we can assume there exists $\sigma\in \mathcal{S}$ such that $\hat{\sigma}\subset \sigma$. 

According to Proposition \ref{gammafiniteclass} and Proposition \ref{fanweakfanequivalent}, there exists some $\tau_0\in \mathcal{S}$ and a finite set $\{\gamma_1,...,\gamma_L\}\in \Gamma$ such that $\hat{\tau}\subset \tau=\mathrm{Ad}_{\gamma_0}\tau_0$ for some $\gamma_0\in \Gamma$, and $\gamma_0$ induce the same class as $\gamma_l$ in $Z(\sigma)\backslash G_{\mathbb{C}}/Z(\tau_0)$ for some $1\leq l\leq L$. Since replacing $\gamma_0$ by anything in its own $(Z(\sigma), Z(\tau_0))$-coset does not change $\sigma\cap \mathrm{Ad}_{\gamma_0}\tau_0$, it is enough to consider when $(\sigma, \mathrm{Ad}_{\gamma_0}\tau_0)$ is one of the cases in \eqref{allpossiblebusters}.

By the construction of Lemma \ref{simplicialfanconstruction}, support of the (finitely many) cones in \eqref{allpossiblebusters} also supports the simplicial cone complex $\{\tau_j \ | \ 1\leq j\leq M\}$ which is a fan, and $\hat{\Sigma}$ is taken to be the union of $\{\tau_j \ | \ 1\leq j\leq M\}$ and their $\mathrm{Ad}_{\Gamma}$-translations. As a consequence $\hat{\sigma},  \hat{\tau}\in \{\tau_j \ | \ 1\leq j\leq M\}$, but $\{\tau_j \ | \ 1\leq j\leq M\}$ by construction is a fan, therefore we must have $\hat{\sigma}=\hat{\tau}$ and the claim follows.

\subsection{Proof of Claim \ref{basicclaim2}}

Suppose now $\hat{\sigma}, \hat{\tau}\in \mathcal{T}$, and for some $\gamma_0\in \Gamma$ we have $\hat{\sigma}\cap \mathrm{Ad}_{\gamma_0}\hat{\tau}\neq \emptyset$ and $F^{\bullet}\in \tilde{B}(\hat{\sigma}, \Phi)\cap \tilde{B}(\mathrm{Ad}_{\gamma_0}\hat{\tau}, \Phi)\neq \emptyset$. We need to prove if this case shows up, we must have $\hat{\sigma}= \mathrm{Ad}_{\gamma_0}\hat{\tau}$.

Let $\sigma$ and $\tau_0$ be the unique cones in $\mathcal{S}$ containing $\hat{\sigma}$ and some $\hat{\tau}$ respectively, then $\sigma\cap \mathrm{Ad}_{\gamma_0}\tau_0\neq \emptyset$. We split into $2$ cases. 

\subsubsection{Case 1} Both $(\hat{\sigma}, F^{\bullet})$ and $(\mathrm{Ad}_{\gamma_0}\hat{\tau}, F^{\bullet})$ give type I degenerations. In this case we have an analog of Proposition \ref{fanweakfanequivalent} as follows. 

\begin{prop}\label{fanweakfanequivalent2}
    Suppose $\sigma, \tau$ are two nilpotent cones which could polarize type I LMHS of weight $3$ Calabi-Yau type. If $\sigma\cap \tau\neq \emptyset$, then $\tilde{B}(\sigma, \Phi)=\tilde{B}(\tau, \Phi)\neq \emptyset$.
\end{prop}

\begin{proof}
    A complete analog.
\end{proof}

Therefore, $\tilde{B}(\sigma, \Phi)\cap \tilde{B}(\mathrm{Ad}_{\gamma_0}(\tau_0), \Phi)\neq \emptyset$ hence by the construction of $\mathcal{S}$, there exists $\mathrm{Ad}_{\gamma_1}(\tau_0)\in \mathcal{S}$ such that $\bar{\gamma}_0=\bar{\gamma}_1$ in $Z(\sigma)\backslash G_{\mathbb{C}}/Z(\tau_0)$. This means 
\begin{equation}
   \hat{\sigma}\cap \mathrm{Ad}_{\gamma_0}\hat{\tau} = \hat{\sigma}\cap \mathrm{Ad}_{\gamma_1}\hat{\tau},
\end{equation}
and by the construction of $\mathcal{T}$, we must have $\hat{\sigma}=\mathrm{Ad}_{\gamma_0}\hat{\tau}$.

\subsubsection{Case 2} Both $(\hat{\sigma}, F^{\bullet})$ and $(\mathrm{Ad}_{\gamma_0}\hat{\tau}, F^{\bullet})$ give type IV degenerations. In this case, since $W_{\mathfrak{g}}(\sigma)_{-2}$ is not a subset of $F_{\mathfrak{g}}^{-1}$, an analog of Proposition \ref{fanweakfanequivalent2} does not hold anymore. 

Consider the $4L+4$ cones given by setting $(\sigma_i,\sigma_j)$ be any ordered pairs from $\sigma$ and $\tau_0$ in \eqref{allpossiblebusters}. Proceeding as Lemma \ref{simplicialfanconstruction} with the subset $\{\sigma, \tau_0\}\subset \mathcal{S}$ instead, we obtained a polyhedral complex $\mathcal{T}(\sigma, \tau_0)$ satisfies the following conditions:
\begin{description}
    \item[(i)] $|\mathcal{T}(\sigma, \tau_0)|\subset |\mathcal{T}|$.
    \item[(ii)] Any $\tau\in \mathcal{T}(\sigma, \tau_0)$ is a disjoint union of cones in $\mathcal{T}$.
\end{description}
Let $\tilde{\mathcal{T}}$ and $\tilde{\mathcal{T}}(\tilde{\sigma}, \tilde{\tau_0})$ be the projections of $\mathcal{T}$ and $\mathcal{T}(\sigma, \tau_0)$ on $
\tilde{H}$ defined by \eqref{quotientvectorspaceH} via \eqref{projectionoftypeIVcones}, with $\hat{\sigma}_-$ and $\hat{\tau}_-$ being the images of $\hat{\sigma}$ and $\hat{\tau}$ respectively. By condition (ii) above there exist $\hat{\sigma}_1$ and $\hat{\tau}_1$ in $\tilde{\mathcal{T}}(\tilde{\sigma}, \tilde{\tau_0})$ containing $\hat{\sigma}_-$ and $\hat{\tau}_-$ respectively. 

Since $\tilde{\mathcal{T}}(\tilde{\sigma}, \tilde{\tau_0})$ is now a collection of nilpotent cones polarizing LMHS of classical weight $1$ type on $\tilde{H}$, by the previous proof of Claim \ref{basicclaim} we know $\mathrm{Ad}(\tilde{\Gamma})\tilde{\mathcal{T}}(\tilde{\sigma}, \tilde{\tau_0})$ is a fan. Therefore, $\hat{\sigma}_-\cap \mathrm{Ad}_{\tilde{\gamma}}\hat{\tau}_-\neq \emptyset$ implies $\hat{\sigma}_1=\mathrm{Ad}_{\tilde{\gamma}}\hat{\tau}_1$, and again by the condition (ii) above, $\hat{\sigma}_-=\mathrm{Ad}_{\tilde{\gamma}}\hat{\tau}_-$. Together with the assumption $\tilde{B}(\hat{\sigma}, \Phi)\cap \tilde{B}(\mathrm{Ad}_{\gamma_0}\hat{\tau}, \Phi)\neq \emptyset$, Lemma \ref{typeIVtricklemma} implies $\hat{\sigma}= \mathrm{Ad}_{\gamma_0}\hat{\tau}$.

%% file: Sec8_appendix.tex
\section{Appendix: Kato-Usui's theory: Overview and modification}

The purpose of this appendix is briefly reviewing Kato-Usui's theory and sketch the proof of Theorems \ref{katousuimainthmA} and \ref{katousuimainthmB}. 
\subsection{Kato-Usui's theory revisit}
\subsubsection{Basic constructions}
We begin with the local constructions following \cite[Chap. 3]{KU08}. Given a rational nilpotent cone $\sigma\subset \mathfrak{g}_{\mathbb{Q}}$, let
\begin{align}
&\mathrm{toric}_{\sigma}:=\mathrm{Spec}(\mathbb{C}[\Gamma(\sigma)^{\vee}])_{\mathrm{an}}\\
\nonumber &\mathrm{torus}_{\sigma}:=\mathrm{Spec}(\mathbb{C}[(\Gamma(\sigma)^{\mathrm{gp}})^{\vee}])_{\mathrm{an}}
\end{align}
be the toric variety and torus embedding associated to $\sigma_{\mathbb{R}}$. For any $q\in \mathrm{toric}_{\sigma}$ we associate the cone $\sigma(q)$ which is the face of $\sigma$ corresponds to the torus orbit $q$ lies in via the orbit-cone correspondence\footnote{See for example, \cite[Chap. 3]{CLS11}}. Regarding (evaluating at) $q$ as a regular function over $\mathbb{C}(\Gamma(\sigma)^{\vee})$, it also defines a class $[q]\in \frac{\sigma_{\mathbb{C}}}{\sigma(q)_{\mathbb{C}}+\mathrm{log}\Gamma(\sigma)^{\mathrm{gp}}}$, of which we denote any lift as $\mathrm{log}_{\sigma}q$.\\

Following \cite[Chap. 3]{KU08}, define
\begin{align}
\check{E}_{\sigma}&:=\mathrm{toric}_{\sigma}\times D\\
\nonumber \tilde{E}_{\sigma}&:=\{(q,F^{\bullet})\in \check{E}_{\sigma} \ | \ NF^{\bullet}\subset F^{\bullet-1} \ \forall  N\in \sigma(q)\}\\
\nonumber E_{\sigma}&:=\{(q,F^{\bullet})\in \tilde{E}_{\sigma} \ | \ (\sigma(q), \mathrm{exp}(\mathrm{log}_{\Gamma_{\sigma}}q)F^{\bullet}) \hbox{ is a } \sigma(q)-\hbox{nilpotent orbit } \}
\end{align}
Consider the map:
\begin{equation}\label{sigmatorsormap}
    \Theta_{\sigma}: E_{\sigma}\rightarrow \Gamma(\sigma)^{\mathrm{gp}}\backslash D_{\sigma},
\end{equation}
which is the quotient of the action:
\begin{equation}
    a\in \sigma_{\mathbb{C}}: E_{\sigma}\rightarrow E_{\sigma}, \ (a, (q,F^{\bullet}))\rightarrow (e(a)q, e^{-a}F^{\bullet}).
\end{equation}
This map realizes $E_{\sigma}$ as a $\sigma_{\mathbb{C}}$-torsor over $\Gamma(\sigma)^{\mathrm{gp}}\backslash D_{\sigma}$. 

Let $|\mathrm{toric}|_{\sigma}$ be the analytic closure of $\mathbb{R}$-points in $\mathrm{toric}_{\sigma}$, then we can similarly define:
\begin{align}
    \check{E}^{\#}_{\sigma}&:= |\mathrm{toric}|_{\sigma},\\
    \nonumber E^{\#}_{\sigma}&:=E_{\sigma}\cap \check{E}^{\#}_{\sigma},\\
    \nonumber \Theta^{\#}_{\sigma}&: E^{\#}_{\sigma}\rightarrow D^{\#}_{\sigma}.
\end{align}
if we replace $\sigma_{\mathbb{C}}$ by $i\sigma_{\mathbb{R}}$, and nilpotent orbits by nilpotent $i$-orbits. There is a commutative diagram
\begin{equation}\label{diagramoftorsors}
\begin{tikzcd}
  E^{\#}_{\sigma} \arrow[d] \arrow[d] \arrow[r] & E_{\sigma} \arrow[d]\\
D^{\#}_{\sigma} \arrow[r] & \Gamma(\sigma)^{\mathrm{gp}}\backslash D_{\sigma}
\end{tikzcd}
\end{equation}
which realizes $E^{\#}_{\sigma}$ as an $i\sigma_{\mathbb{R}}$-torsor over $D^{\#}_{\sigma}$.

\begin{definition}[\cite{KU08}, Chap. 2]
    For a subset $X\subset S$ where $S$ is a complex analytic space, define the strong topology on $X$ as follows: It is the finest topology such that for any morphism of complex analytic spaces $f: T\rightarrow S$ with $f(T)\subset X$, the (set-theoretic) map $f: T\rightarrow X$ is continuous.
\end{definition}
\begin{remark}
    Clearly the strong topology is equivalent to or finer than the subspace topology for $X\subset S$, and they are not equivalent in general. For such an example, see \cite[Sec. 3.1.3]{KU08}.
\end{remark}

Following \cite[Chap. 2]{KU08}, we put the strong topology on $E_{\sigma}\subset \check{E}_{\sigma}$, and its induced subspace topology on $E^{\#}_{\sigma}$, then the quotient topology on $D^{\#}_{\Sigma}$ (resp. $\Gamma(\sigma)^{\mathrm{gp}}\backslash D_{\sigma}$) regarding the $i\sigma_{\mathbb{R}}$ (resp. $\sigma_{\mathbb{C}}$) torsor map, and then the strongest topology on $D^{\#}_{\Sigma}$ (resp. $\Gamma\backslash D_{\Sigma}$) such that the inclusion maps:
\begin{align}\label{localembeddingtoglobal}
    &D^{\#}_{\sigma} \hookrightarrow D^{\#}_{\Sigma}\\
   \nonumber &\Gamma(\sigma)^{\mathrm{gp}}\backslash D_{\sigma} \hookrightarrow \Gamma\backslash D_{\Sigma}
\end{align}
are continuous for every $\sigma\in \Sigma$. 
We also need the definition of valuative spaces as follows:
\begin{definition}[Valuative submonoids]
\
\begin{description}
\item[(i)]$\mathcal{V}:=$ $\{(A, V) \ |$  A is a $\mathbb{Q}$-linear subspace of $\mathfrak{g}_\mathbb{Q}$ consisting of mutually commutative nilpotent elements, $V$ is a submonoid of $A^*:=\mathrm{Hom}_\mathbb{Q}(A, \mathbb{Q})$ such that $V\cap (-V)=\{0\}$ and $V\cup (-V)=A^*$\}.
\item[(ii)]Given $(A,V)\in \mathcal{V}$, let $\mathcal{F}(A,V)$ be the set of rational nilpotent cones $\sigma \subset \mathfrak{g}_{\mathbb{R}}$ such that $\sigma_\mathbb{R}=A_\mathbb{R}$ and $(\sigma \cap A)^{\vee}:= \{h \in A^* \ | \ h(\sigma \cap A) \subset \mathbb{Q}_{\geq 0}\} \subset V$.
\end{description}
\end{definition}
\begin{definition}[Valuative spaces]
\
\begin{description}
\item[(i)] $\check{D}_{\mathrm{val}}:=\{(A,V,Z) \ | \ (A,V)\in \mathcal{V}$ and $Z$ is an $\mathrm{exp}(A_{\mathbb{C}})$-orbit in $\check{D}$\}
\item[(ii)] ${D}_{\mathrm{val}}:=\{(A,V,Z) \ | \ (A,V,Z)\in \check{D}_{\mathrm{val}}$  and there exists $\sigma \in \mathcal{F}(A,V)$ such that $Z$ is a $\sigma$-nilpotent orbit\}
\end{description}
\end{definition}

\begin{definition}[Relative valuative spaces]
\
\begin{description}
\item[(i)] Suppose $\Sigma$ is a fan in $\mathfrak{g}_{\mathbb{Q}}$. For $(A,V) \in \mathcal{V}$, let 
\begin{center}
    $X_{A,V,\Sigma}:=\{\sigma \in \Sigma \ | \ \sigma \cap A_{\mathbb{R}} \in \mathcal{F}(A,V)\}$
\end{center}
By \cite{KU08}, whenever $X_{A,V,\Sigma}$ is not empty, there is a minimal element $\sigma_0 \in X_{A,V,\Sigma}$.  
\item[(ii)] Suppose $\Sigma$ is a fan in $\mathfrak{g}_{\mathbb{Q}}$, define:
\begin{center}
    $D_{\Sigma, \mathrm{val}}:=\{(A,V,Z) \in \check{D}_{\mathrm{val}} \ | \ X_{A,V,\Sigma}$ is non-empty, and $\mathrm{exp}(\sigma_{0,\mathbb{C}})Z$ is a $\sigma_0$-nilpotent orbit\}
\end{center}

\end{description}

\end{definition}

We note that there are similar definitions for the valuative space of (nilpotent) $i$-orbits $D^{\#}_{\mathrm{val}}, D^{\#}_{\Sigma, \mathrm{val}}$. The topology on valuative spaces are defined similarly to their normal counterparts, the details can be found at \cite[Sec. 5.3]{KU08}.

\subsubsection{Main results}
With all notations introduced before, suppose now $\Sigma$ is a fan. We are now ready to state Kato-Usui's main results.

\begin{theorem}[\cite{KU08}, Sec. 7.2]\label{localhausdorff}
    The action of $\sigma_{\mathbb{C}}$ ($i\sigma_{\mathbb{R}}$) on $E_{\sigma}$ ($E_{\sigma}^{\#}$) is proper and free, and as corollary, $D_{\sigma}^{\#}$ and $\Gamma(\sigma)^{\mathrm{gp}}\backslash D_{\sigma}$ are Hausdorff.
\end{theorem}

\begin{theorem}[\cite{KU08}, Sec. 7.3]\label{locallogmanifold}
    $\Gamma(\sigma)^{\mathrm{gp}}\backslash D_{\sigma}$ admits a structure of logarithmic manifold under which $E_{\sigma}\rightarrow \Gamma(\sigma)^{\mathrm{gp}}\backslash D_{\sigma}$ is a $\sigma_{\mathbb{C}}$-torsor of logarithmic manifolds. Moreover, 
    \begin{equation}
        (\Gamma(\sigma)^{\mathrm{gp}}\backslash D_{\sigma})^{\mathrm{log}}\cong \Gamma(\sigma)^{\mathrm{gp}}\backslash D_{\sigma}^{\#}.
    \end{equation}
\end{theorem}

These are results regarding the local structures, or say when the fan $\Sigma$ is just a cone $\sigma$ and its faces. Passing to the global structure, we have the following:

\begin{theorem}[\cite{KU08}, Sec. 7.3-7.4]\label{local2globalku}
Assuming $\Gamma$ is neat, then $\Gamma\backslash D_{\Sigma}$ admits a structure of logarithmic manifolds as well as Hausdorff topological space, for which the map:
\begin{equation}
    \Gamma(\sigma)^{\mathrm{gp}}\backslash D_{\sigma}\rightarrow \Gamma\backslash D_{\Sigma}
\end{equation}
is a local homeomorphism for any $\sigma\in \Sigma$.
\end{theorem}
\begin{remark}
 The precise defition of logarithmic manifolds can be found in \cite[Sec. 3.5]{KU08}. Roughly speaking, a logarithmic manifold is locally isomorphic to an analytic subset of some fs logarithmic analytic space (defined in \cite[Chap. 2]{KU08}) given by the vanishing locus of a set of logarithmic $1$-forms.   
\end{remark}

Theorems \ref{localhausdorff} and \ref{locallogmanifold} are about the local model $\Gamma(\sigma)^{\mathrm{gp}}\backslash D_{\sigma}$, for which the proof can be found at \cite[Sec. 7.1-7.2]{KU08}. In this paper we will briefly outline the idea of proving Theorem \ref{local2globalku}, i.e. how we patch the local logarithmic manifold structures given by $\Gamma(\sigma)^{\mathrm{gp}}\backslash D_{\sigma}$ to a global one on $\Gamma\backslash D_{\Sigma}$ via a (weak) fan $\Sigma$.

\subsubsection{Proving main results: First step} \cite[Thm. 7.3.2]{KU08} asserts that the inclusion maps $D_{\sigma}^{\#}\hookrightarrow D_{\Sigma}^{\#}$ as well as $D_{\sigma, \mathrm{val}}^{\#}\hookrightarrow D_{\Sigma, \mathrm{val}}^{\#}$ are open maps, and the canonical map:
\begin{equation}
    D_{\Sigma, \mathrm{val}}^{\#}\rightarrow D_{\Sigma}^{\#}
\end{equation} 
is a proper map between Hausdorff topological spaces. Particularly, the first argument uses the fact that if $\Sigma$ is a fan, for any $\sigma, \tau\in \Sigma$ we have $D_{\sigma}^{\#}\cap D_{\tau}^{\#}=D_{\sigma\cap \tau}^{\#}$. Hence by the definition of the topology on $D_{\sigma}^{\#}$, it suffices to prove the case when $\Sigma=\{\tau \ \hbox{and its faces}\}$ and $\sigma\leq \tau$. The second argument comes from the commutative diagram:
\begin{equation}
\begin{tikzcd}
E^{\#}_{\sigma} \arrow[d] \arrow[r] & E^{\#}_{\sigma, \mathrm{val}} \arrow[d] \\
D^{\#}_{\sigma} \arrow[r] & D^{\#}_{\sigma, \mathrm{val}}  
\end{tikzcd}
\end{equation}
where the vertical arrows are $i\sigma_{\mathbb{R}}$-torsor maps and the horizontal arrows are proper maps by \cite[Thm. 5.3.8]{KU08}.

\begin{remark}\label{fanweakfanreplacement}
    If we assume $\Sigma$ is a weak fan (instead of a fan), by \cite[Thm. 4.3.1]{KNU10}, the condition $D_{\sigma}^{\#}\cap D_{\tau}^{\#}=D_{\sigma\cap \tau}^{\#}$ should be modified by:
    \begin{equation}
        D_{\sigma}^{\#}\cap D_{\tau}^{\#}=\cup_{\alpha}D_{\alpha}^{\#}
    \end{equation}
    where $\alpha$ ranges among all common faces of $\sigma$ and $\tau$. The rest of proof is identical to \cite[Thm 7.3.2]{KU08}.
\end{remark} 

\subsubsection{Proving main results: Second step} 
Once we prove the natural map 
\begin{equation}\label{localchartsofku}
    \Gamma(\sigma)^{\mathrm{gp}}\backslash D_{\sigma}\rightarrow \Gamma\backslash D_{\Sigma}
\end{equation}
is a local homeomorphism, the logarithmic manifold structure on $\Gamma\backslash D_{\Sigma}$ will be verified by regarding every $\Gamma(\sigma)^{\mathrm{gp}}\backslash D_{\sigma}$ as a local chart.

The central idea to prove the map \eqref{localchartsofku} is a local homeomorphism is applying \cite[Lemma 7.4.7]{KU08}. For convenience we re-state it in terms of spaces we are considering.

\begin{prop}\label{extensionlemma}
    The following conditions hold:
    \begin{description}
        \item[(i)] The map $D_{\Sigma}^{\#}\rightarrow \Gamma\backslash D_{\Sigma}^{\#}$ is a local homeomorphism between Hausdorff topological spaces.
        \item[(ii)] The map $\Gamma(\sigma)^{\mathrm{gp}}\backslash D_{\sigma}^{\#}\rightarrow \Gamma(\sigma)^{\mathrm{gp}}\backslash D_{\sigma}$ is proper.
        \item[(iii)] For any $(\tau, Z)\in D_{\sigma}$ and $\gamma\in \Gamma$, if $\gamma(\tau, Z)$ and $(\tau, Z)$ lie in the same $\Gamma(\sigma)$-orbit, then $\gamma\in \Gamma(\sigma)$.
    \end{description}
    Moreover, these properties imply \eqref{localchartsofku} is a local homeomorphism.
\end{prop}

We first sketch the proof of conditions (ii) and (iii). For (ii), this comes from the theorem on local charts \eqref{locallogmanifold} and basic properties of logarithmic analytic spaces. For (iii), it easily follows from the following proposition:

\begin{prop}[\cite{KU08}, Prop. 7.4.3]\label{nilporbitstabilizer}
    Assume $\Gamma$ is neat and $(\sigma, Z)\in D_{\Sigma}^{\#}$ $(\mathrm{resp. } \  D_{\Sigma})$, then $\mathrm{Stab}_{\Gamma}(\sigma, Z)=\{\mathrm{Id}\}$ $(\mathrm{resp. } \ \Gamma(\sigma)^{\mathrm{gp}})$.
\end{prop}

The most difficult part is proving the condition (i) in Proposition \ref{extensionlemma}. This follows from the following theorem.

\begin{theorem}[\cite{KU08}, Thm. 7.4.2]\label{fundamentaltheoremoniorbits}
\
\begin{description}
    \item[(i)] The action of $\Gamma$ on $D_{\Sigma}^{\#}$ and $D_{\Sigma, \mathrm{val}}^{\#}$ are proper, and the quotient spaces $\Gamma\backslash D_{\Sigma}^{\#}$ and $\Gamma\backslash D_{\Sigma, \mathrm{val}}^{\#}$ are Hausdorff;
    \item[(ii)] Assume $\Gamma$ is neat, the canonical maps $D_{\Sigma}^{\#}\rightarrow \Gamma\backslash D_{\Sigma}^{\#}$ and $D_{\Sigma, \mathrm{val}}^{\#}\rightarrow \Gamma\backslash D_{\Sigma, \mathrm{val}}^{\#}$
    are local homeomorphisms.
\end{description}
    
\end{theorem}

Suppose (i) of Theorem \ref{fundamentaltheoremoniorbits} holds. Since $\Gamma$ is a discrete group acting properly and freely on the Hausdorff space $D_{\Sigma}^{\#}$, it is well-known the quotient space $\Gamma\backslash D_{\Sigma}^{\#}$ is Hausdorff and the quotient map is a local homeomorphism. The same argument holds for the valuative space, therefore (ii) of Theorem \ref{fundamentaltheoremoniorbits} holds.
\subsubsection{Proving main results: Third step}
To prove (i) of Theorem \ref{fundamentaltheoremoniorbits}, we need to formulate theory regarding the space of $\mathrm{SL}_2$-orbits $D_{\mathrm{SL}_2}$. Regarding the general theory of $\mathrm{SL}_2$-orbits, see for example \cite{KU02} or \cite[Chap. 5-6]{KU08}.

We shall begin with the classical $\mathrm{SL}_2$-orbit theorem, which was proved in \cite{Sch73} for $1$-variable case and \cite{CKS86} for general case. We sketch these classical results in a form convenient for our use.

For the field $\mathbb{F}= \mathbb{Q}, \mathbb{R}$ or $\mathbb{C}$, denote
\begin{equation}
    \mathbf{n}_-:=\begin{pmatrix}
        0 & 1\\
        0 & 0
    \end{pmatrix}, \ 
    \mathbf{y}:=\begin{pmatrix}
        -1 & 0\\
        0 & 1
    \end{pmatrix}, \ 
    \mathbf{n}_+:=\begin{pmatrix}
        0 & 0\\
        1 & 0
    \end{pmatrix}.
\end{equation}
be the standard generator of $\mathfrak{sl}_2(\mathbb{F})$ with the relations:
\begin{equation}
    [\mathbf{y},\mathbf{n}_{\pm}]=\pm 2\mathbf{n}_{\pm}, \ [\mathbf{n}_+, \mathbf{n}_-]=\mathbf{y}.
\end{equation}
Using Cayley transform we can associate a natural $\mathbb{F}$-Hodge structure of weight $0$ on $\mathfrak{sl}_2(\mathbb{F})$ as follows:
\begin{align}
    (\mathfrak{sl}_2(\mathbb{C}))^{-1,1}&=\overline{(\mathfrak{sl}_2(\mathbb{C}))^{1,-1}}=\mathbb{C}(i\mathbf{y}+\mathbf{n}_++\mathbf{n}_-),\\
   \nonumber (\mathfrak{sl}_2(\mathbb{C}))^{0,0}&=\mathbb{C}(\mathbf{n}_+-\mathbf{n}_-).
\end{align}

Consider an $\mathfrak{sl}_2(\mathbb{F})$-representation:
\begin{equation}
    \rho: \mathfrak{sl}_2(\mathbb{F})\rightarrow \mathfrak{g}_{\mathbb{F}}
\end{equation}
which has a lift to the algebraic group level:
\begin{equation}
    \tilde{\rho}: \mathrm{SL}_2(\mathbb{F})\rightarrow G_{\mathbb{F}}.
\end{equation}
We denote
\begin{equation}
    [N, Y, N^+]:=[\rho(n_-), \rho(y), \rho(n_+)]
\end{equation}
as the associated $\mathfrak{sl}_2$-triple. We also say a pair of $\mathfrak{sl}_2$-triples, namely 
\begin{equation}
    [N_i, Y_i, N^+_i], \ i=1,2
\end{equation}
\textbf{commute}, if the subalgebras generated by two triples commute.

$\mathfrak{sl}_2$-representations arise in Hodge theory in a natural way. Suppose $(W(N), F^{\bullet}, N)$ is an $\mathbb{R}$-split mixed Hodge structure polarized by $N\in \mathfrak{g}_{\mathbb{Q}}$, let $Y\in \mathfrak{g}_{\mathbb{R}}$ be the semisimple element grading the LMHS, then there exists a unique $\mathfrak{sl}_2$-representation $\rho$ such that
\begin{equation}\label{sl2morphism}
    \rho(n_-)=N, \ \rho(y)=Y.
\end{equation}
We complete it into an $\mathfrak{sl}_2$-triple $[N, Y, N^+]$. There is a very special class of $\mathfrak{sl}_2$-representations which induce horizontal $\mathrm{SL}_2$-orbits, see for example \cite{Rob15}.
\begin{definition}
 Fix a point $F^{\bullet}\in D$, it induces a weight-$0$ Hodge structure $F_{\mathfrak{g}}^{\bullet}$ on $\mathfrak{g}=\mathrm{End}(V)$. We say $\rho$ gives a horizontal $\mathrm{SL}_2$-orbit at $F^{\bullet}$ if the following holds:
\begin{description}
    \item[(i)] $\tilde{\rho}$ is defined on $\mathbb{R}$;
    \item[(ii)] $\rho$ is a morphism of Hodge structures with respect to $F_{\mathfrak{g}}^{\bullet}$. Equivalently, 
    \begin{equation}
         \rho(\mathfrak{sl}_2(\mathbb{C}))^{-k,k})\subset F_{\mathfrak{g}}^{-k,k.}
    \end{equation} 
\end{description}
We also say $\rho$ gives a horizontal $\mathrm{SL}_2$-orbit if it gives one at some point $F^{\bullet}\in D$.
\end{definition}
\begin{prop}
    Let $(W(N)[-l], F^{\bullet}, N)$ be an $\mathbb{R}$-split mixed Hodge structure polarized by $N$ and $\rho$ be an $\mathfrak{sl}_2$-representation as \eqref{sl2morphism}, then:
    \begin{description}
    \item[(i)] $\rho$ is a morphism of Hodge structures at $\exp(iN)F^{\bullet}$;
    \item[(ii)] $\rho$ gives $\tilde{\rho}$ which is defined over $\mathbb{R}$, and 
        \begin{equation}
            \hat{\rho}: \mathbb{P}^1\rightarrow \check{D}
        \end{equation}
        an equivariant $\mathrm{SL}_2$-embedding by $\hat{\rho}(gi)=\tilde{\rho}(g)\exp(iN)F^{\bullet}$.
    \end{description}
\end{prop}

Now we are ready to state the $\mathrm{SL}_2$-orbit theorem of \cite{Sch73} and \cite{CKS86}. Consider a nilpotent orbit 
\begin{equation}
    (z_1,..,z_n)\xrightarrow{\theta} \exp(\sum_{j=1}^nz_jN_j)F^{\bullet}.
\end{equation}
\cite[Thm. 4.20]{CKS86}, i.e. the multivariable $\mathrm{SL}_2$-orbit theorem can be summarized as follows.
\begin{theorem}[\cite{CKS86}, Thm. 4.20]\label{multivarsl2orbit}
    There exists a unique $\mathbb{R}$-representation:
    \begin{equation}\label{sl2representation}
        \rho: (\mathfrak{sl}_2(\mathbb{C}))^n\rightarrow \mathfrak{g}_{\mathbb{C}}
    \end{equation}
    and its lift:
    \begin{equation}
        \rho: (\mathrm{SL}_2(\mathbb{C}))^n\rightarrow G_{\mathbb{C}}
    \end{equation}
    with the following properties:
    \begin{description}
        \item[(i)] Let 
        \begin{equation}
            [\hat{N}_j, \hat{Y}_j, \hat{N}^+_j], \ 1\leq j\leq n
        \end{equation}
        be the $\mathfrak{sl}_2$-triple associated to the $j$-factor of $\rho$, then these $n$-triples commute;
        \item[(ii)] $W(\sum_{k=1}^{j}N_k)=W(\sum_{k=1}^j\hat{N}_k)$;
    \end{description}
    Moreover, there exists a unique $F_0^{\bullet}\in \check{D}$ such that:
    \begin{description}
       \item[(iii)] $\rho$ gives an $n$-dimensional horizontal $\mathrm{SL}_2$-orbit at $F_0^{\bullet}$;
        \item[(iv)] Denote $\rho_r, \ 1\leq r\leq n $ as the restriction of $\rho$ on the first $r$ factors, then $\rho_r$ is a morphism of Hodge structures at the point $\exp(i\sum_{k=r+1}^{n}\hat{N}_k)F_0^{\bullet}$;
        \item[(v)]For $y_1,...,y_n\in \mathbb{R}$, there exists a real analytic function $g(y_1,...,y_n): \mathbb{R}^n\rightarrow G_{\mathbb{R}}$ such that
        \begin{equation}\label{sl2approximation}
        \exp(i\sum_{j=1}^ny_jN_j)F^{\bullet}=g(y_1,...,y_n)\exp(i\sum_{j=1}^ny_j\hat{N}_j)F_0^{\bullet};
        \end{equation}
        \item[(vi)] Use the convention $y_{n+1}=1$, $g(y_1,...,y_n)$ can be written as a Laurent series of $\frac{y_{j}}{y_{j+1}}, \ 1\leq j\leq n$ with non-positive powers and constant term $1$. Therefore, the two sides of \eqref{sl2approximation} has the same limit as $\frac{y_{j}}{y_{j+1}}\rightarrow \infty, \ 1\leq j\leq n$.
        \item[(vii)] Again use the convention $y_{n+1}=1$, we have:
        \begin{equation}
            \exp(i\sum_{j=1}^ny_j\hat{N}_j)F^{\bullet}=\exp(-\frac{1}{2}\sum_{j=1}^{n}(\log(\frac{y_j}{y_{j+1}})\sum_{k=1}^{j}\hat{Y}_k))\exp(i\sum_{j=1}^n\hat{N_{j}})F_0^{\bullet}.
        \end{equation}
    \end{description}
\end{theorem}

We call $(\rho, \hat{\rho})$ given by Theorem \ref{multivarsl2orbit} as the $\mathrm{SL}_2$-orbit associated to the nilpotent orbit $\theta: (z_1,...,z_n)\rightarrow\exp(\sum_{j=1}^nz_jN_j)F^{\bullet}$. The idea of patching all of these "distinguished $\mathrm{SL}_2$-orbits" to a space leads to the following definition.

\begin{definition}[\cite{KU02}, Definition 3.6]\label{spaceofsl2orbits}
    The space of $\mathrm{SL}_2$-orbits, denoted as $D_{\mathrm{SL}_2}$, be the collection of all $(\rho, \hat{\rho})$ modulo the equivalence relation $\sim$, such that:
\begin{description}
    \item[(i)] $(\rho, \hat{\rho})$ is an $\mathrm{SL}_2$-orbit;
    \item[(ii)] The weight filtrations $W(\sum_{k=1}^j\hat{N}_k)$ are rational for $1\leq j\leq n$;
    \item[(iii)] $(\rho, \hat{\rho})\sim (\rho^{'}, \hat{\rho}^{'})$ if there exists $(t_1,...,t_n)\in \mathbb{R}_{>0}^n$ such that $\hat{Y}_j=\hat{Y}_j^{'}$ and $\hat{N}_j=t_j\hat{N}_j^{'}$.
\end{description}
\end{definition}

\begin{remark}
    In \cite{KU02}, the rank of an $\mathrm{SL}_2$-orbit is defined. In this paper we always assume the rank of an $\mathrm{SL}_2$-orbit equals to the number of independent variables $n$, or equivalently \eqref{sl2representation} is injective.
\end{remark}

The topology of the space $D_{\mathrm{SL}_2}$ is constructed in \cite[Sec. 5.2-5.4]{KU08} and we will omit it here. The point we would like to highlight is the existence of a natural map:
\begin{equation}\label{valuativetosl2}
    \psi: D_{\Sigma, \mathrm{val}}^{\#}\rightarrow D_{\mathrm{SL}_2}
\end{equation}
which sends a (valuative) nilpotent orbit to its associated $\mathrm{SL}_2$-orbit('s equivalent class) by Theorem \ref{multivarsl2orbit}. Well-defineness of this map can be seen from \cite[Thm. 5.4.3]{KU08}. The critical result is the following:
\begin{theorem}[\cite{KU08}, Thm 5.4.4]\label{continuousmapsl2space}
    The map $\psi$ is continuous, and is the unique continuous extension of the identity map on $D$.
\end{theorem}
Now we return to Theorem \ref{fundamentaltheoremoniorbits} (i). By Theorem \ref{continuousmapsl2space}, Proposition \ref{extensionlemma} and the fact that $\Gamma$ acts properly on $D_{\mathrm{SL}_2}$ (see \cite[Thm. 5.2.15]{KU08}), the result follows and hence Proposition \ref{extensionlemma} is proved.
\subsubsection{Final step: Conclusion}
Proposition \ref{extensionlemma} says the local (Hausdorff) topological and logarithmic analytic structure of $\Gamma\backslash D_{\Sigma}$ are completely inherited from those for $\Gamma(\sigma)^{\mathrm{gp}}\backslash D_{\sigma}$, which are proved in Theorems \ref{localhausdorff} and \ref{locallogmanifold}. This certifies Theorem \ref{local2globalku}.

A remark is that the main Theorems \ref{katousuimainthmA} and \ref{katousuimainthmB} remain valid if we only require $\Sigma$ to be a weak fan. This is proved in \cite{KNU10} for the polarized mixed Hodge structure case. 
\subsection{Weak fan of restricted types}

In this subsection we shall go through the same process as the last section, but for a weak fan of restricted type (Definition \ref{typeweakfan}). In particular, we shall show Theorems \ref{kumainthmA'} and \ref{kumainthmB'}. 

Continue using notations from Section 3 and the last section, and let $\Phi$ be an index set containing a certain selection of LMHS types. In this section we shall always assume $\Sigma$ is a weak fan of type $\Phi$. Define:
\small
\begin{equation}
    E_{\sigma, \Phi}:=\{(q,F^{\bullet})\in \tilde{E}_{\sigma} \ | \ (\sigma(q), \mathrm{exp}(\mathrm{log}_{\Gamma_{\sigma}}q)F^{\bullet}) \hbox{ is a } \sigma(q)-\hbox{nilpotent orbit of type $\Phi$}\}.
\end{equation}
\normalsize
Clearly $E_{\sigma, \Phi}\subset E_{\sigma}$ and we endow it with the subspace topology. The map \eqref{sigmatorsormap} has the restriction
\begin{equation}
    \Theta_{\sigma}|_{ E_{\sigma, \Phi}}: E_{\sigma, \Phi}\rightarrow \Gamma(\sigma)^{\mathrm{gp}}\backslash D_{\sigma, \Phi},
\end{equation}
which is clearly a $\sigma_{\mathbb{C}}$-torsor if we endow $\Gamma(\sigma)^{\mathrm{gp}}\backslash D_{\sigma, \Phi}$ the quotient topology. Similarly we can define $E^{\#}_{\sigma, \Phi}$ with the subspace topology and the $i\sigma_{\mathbb{R}}$-torsor
\begin{equation}
    \Theta_{\sigma}^{\#}|_{ E^{\#}_{\sigma, \Phi}}: E^{\#}_{\sigma, \Phi}\rightarrow D^{\#}_{\sigma, \Phi}.
\end{equation}
and hence the diagram \eqref{diagramoftorsors} can be enlarged to
\begin{equation}\label{enlargeddiagramoftorsors}
\begin{tikzpicture}[scale=1.2]

    \node at (0,3) {$E_{\sigma}^{\#}$};
    \node at (0,0) {$D_{\sigma}^{\#}$};
    \node at (6,0) {$\Gamma(\sigma)^{\mathrm{gp}}\backslash D_{\sigma}$};
    \node at (6,3) {$E_{\sigma}$};

    \node at (2,1) {$D^{\#}_{\sigma, \Phi}$};
    \node at (4,2) {$E_{\sigma, \Phi}$};
    \node at (4.5,0.95) {\small $\Gamma(\sigma)^{\mathrm{gp}}\backslash D_{\sigma, \Phi}$\normalsize};
    \node at (2,2) {$E^{\#}_{\sigma, \Phi}$};
    
    \draw [->](0.5,0) -- (5,0) [line width = 0.3pt];
    \draw [->](0.5,3) -- (5.5,3) [line width = 0.3pt];
    \draw [->](0,2.5) -- (0,0.5) [line width = 0.3pt] node[midway,left] {$\Theta^{\#}_{\sigma}$};
    \draw [->](6,2.5) -- (6,0.5) [line width = 0.3pt] node[midway, right] {$\Theta_{\sigma}$};

    \draw [->](2.5,1) -- (3.5,1) [line width = 0.3pt];
    \draw [->](2.5,2) -- (3.5,2) [line width = 0.3pt];
    \draw [->](2,1.7) -- (2,1.3) [line width = 0.3pt] node[midway, left] {\small{$\Theta^{\#}_{\sigma}|_{ E^{\#}_{\sigma, \Phi}}$}};
    \draw [->](4,1.7) -- (4,1.3) [line width = 0.3pt] node[midway, right] {\small{$\Theta_{\sigma}|_{ E_{\sigma, \Phi}}$}};

    \draw [->](1.5,2.3) -- (0.5,2.7) [line width = 0.3pt];
    \draw [->](4.5,2.3) -- (5.5,2.7) [line width = 0.3pt];
    \draw [->](1.5,0.7) -- (0.5,0.3) [line width = 0.3pt];
    \draw [->](4.5,0.7) -- (5.5,0.3) [line width = 0.3pt];

    \end{tikzpicture}
\end{equation}
Now we prove Theorem \ref{kumainthmA'} for the local case where $\Sigma=\{\sigma \ \hbox{and its faces}\}$. Following \cite[Sec. 7.1-7.3]{KU08}, we have to show:
\begin{description}
    \item[(i)] $E_{\sigma, \Phi}$ admits a structure of locally analytically constructible space.
    \item[(ii)] The action of $\sigma_{\mathbb{C}}$ (resp. $i\sigma_{\mathbb{R}}$) on $E_{\sigma, \Phi}$ (resp. $E^{\#}_{\sigma, \Phi}$) is proper and free.
\end{description}
Indeed, we already know (ii) holds as the restricted maps $\Theta_{\sigma}|_{ E^{\#}_{\sigma, \Phi}}$ and $\Theta^{\#}_{\sigma}|_{ E^{\#}_{\sigma, \Phi}}$ clearly give torsors in the category of Hausdorff topological spaces. We shall now show (i).

Notice that for any $(q,F^{\bullet})\in E_{\sigma}$ and some open neighborhood $(q,F^{\bullet})\in U\subset \tilde{E}_{\sigma}$, the proof of \cite[Theorem 3.5.10]{KU08} shows as a locally analytically constructible space (Definition \ref{deflocallyanalyticallyconstr}), if for any face $\tau\leq \sigma$ we denote $T_{\tau}\subset $ as the associated torus orbit in $\mathrm{toric}_{\sigma}$, then the analytically constructible strata for $U\subset \tilde{E}_{\sigma}$ are given as follows:
\begin{align}
    A_{\tau}&:=(T_{\tau}\times\check{D})\cap U,\\
    \nonumber B_{\tau}&:= \{(q,F^{\bullet})\in A_{\tau} \ | \ NF^{\bullet}\subset F^{\bullet-1}, \ \forall N\in \tau\}.
\end{align}
By shrink $U$ if necessary (to fit the positivity condition), we obtain a description of the locally analytically constructible structure on $E_{\sigma}$. 

Since for any fixed weight filtration $W=W(\tau)$, the restriction of MHS type $\Phi$ on $(W, F^{\bullet})$ is given by the numerical conditions on (graded) Hodge numbers:
\begin{equation}
    \mathrm{Gr}_{W_k}f^p:=\mathrm{dim}_{\mathbb{C}}(W^{k}\cap F^p)-\mathrm{dim}_{\mathbb{C}}(W^{k-1}\cap F^p),
\end{equation}
which are clearly complex analytic conditions on $A_{\tau}$. Therefore, if we replace the strata $B_{\tau}$ by:
\small
\begin{equation}
    B_{\tau, \Phi}:= \{(q,F^{\bullet})\in A_{\tau} \ | \ NF^{\bullet}\subset F^{\bullet-1}, \ \forall N\in \tau, \ \mathrm{Gr}_{W(\tau)_k}f^p \hbox{ are prescribed by} \ \Phi\},
\end{equation}
\normalsize
we complete the description of $E_{\sigma, \Phi}$ as a locally analytically constructible space. The condition (i) is verified.

With (i) and (ii) established, by applying \cite[Lemma 7.3.3]{KU08}\footnote{Though it is stated with the category of either Hausdorff topological spaces or logarithmic manifolds, it also holds for locally analytically constructible spaces.} and \cite[Sec. 7.3.4-7.3.7]{KU08} verbatim we shall get the desired result, i.e. the space $\Gamma(\sigma)^{\mathrm{gp}}\backslash D_{\sigma, \Phi}$ admits the structure of locally analytically constructible space. Moreover, this structure has two interpretations: Either the one from $E_{\sigma, \Phi}$ using \cite[Lemma 7.3.3]{KU08} or the one obtained from $\Gamma(\sigma)^{\mathrm{gp}}\backslash D_{\sigma, \Phi}\subset \Gamma(\sigma)^{\mathrm{gp}}\backslash D_{\sigma}$.

We turn to the proof of Theorem \ref{kumainthmA'} for a general $\Sigma$. For $D^{\#}_{\sigma, \Phi}\subset D^{\#}_{\sigma}$ and $\Gamma(\sigma)^{\mathrm{gp}}\backslash D_{\sigma, \Phi}\subset \Gamma(\sigma)^{\mathrm{gp}}\backslash D_{\sigma}$ we endow the subspace topology, then endow $D^{\#}_{\Sigma, \Phi}$ and $\Gamma\backslash D_{\Sigma, \Phi}$ the finest topology such that for any $\sigma\in \Sigma$, the maps:
\begin{align}\label{restrictedlocalembeddingtoglobal}
    &D^{\#}_{\sigma, \Phi} \hookrightarrow D^{\#}_{\Sigma, \Phi}\\
    \nonumber &\Gamma(\sigma)^{\mathrm{gp}}\backslash D_{\sigma, \Phi} \hookrightarrow \Gamma\backslash D_{\Sigma, \Phi}
\end{align}
are continuous. We also need to define the restricted valuative space $D_{\Sigma, \Phi, \mathrm{val}}$ and $D^{\#}_{\Sigma, \Phi, \mathrm{val}}$:
\begin{align}
    D_{\Sigma, \Phi, \mathrm{val}}&:=\{(A,V,Z) \in D_{\Sigma, \mathrm{val}} \ | \ Z \  \hbox{is a nilpotent orbit of type} \  \Phi\},\\
   \nonumber D^{\#}_{\Sigma, \Phi, \mathrm{val}}&:=\{(A,V,Z^{\#}) \in D^{\#}_{\Sigma, \mathrm{val}} \ | \ Z \  \hbox{is a nilpotent} \ i\hbox{-orbit of type} \  \Phi\}.
\end{align}
As before, the first step is to obtain an analog of \cite[Thm. 7.3.2]{KU08}. More precisely, $D^{\#}_{\sigma, \Phi}\hookrightarrow D^{\#}_{\Sigma, \Phi}$ and $D^{\#}_{\sigma, \Phi, \mathrm{val}}\hookrightarrow D^{\#}_{\Sigma, \Phi, \mathrm{val}}$ are open maps. We also have the map
\begin{equation}
    D^{\#}_{\sigma, \Phi, \mathrm{val}}\rightarrow D^{\#}_{\sigma, \Phi}
\end{equation}
which is a proper map between Hausdorff topological spaces. We note since $\Sigma$ is a weak fan of type $\Phi$, the statement in Remark \ref{fanweakfanreplacement} should be replaced by:
\begin{equation}
        D_{\sigma, \Phi}^{\#}\cap D_{\tau, \Phi}^{\#}=\cup_{\alpha}D_{\alpha, \Phi}^{\#}
\end{equation}
where $\sigma, \tau\in \Sigma$ and $\alpha$ ranges among all common faces of $\sigma$ and $\tau$. Moreover, the fact $\Sigma$ is a weak fan of type $\Phi$ implies there are well-defined canonical maps:
\begin{align}
    D^{\#}_{\Sigma, \Phi, \mathrm{val}}&\rightarrow D^{\#}_{\Sigma, \Phi},\\
    \nonumber D_{\Sigma, \Phi, \mathrm{val}}&\rightarrow D_{\Sigma, \Phi}.
\end{align}
Therefore, by applying the proof of \cite[Thm. 7.3.2]{KU08} verbatim we get the desired results.

Now Theorem \ref{local2globalku}, Proposition \ref{extensionlemma}, \ref{nilporbitstabilizer} and Theorem \ref{fundamentaltheoremoniorbits} can be directly generalized to the restricted case, as long as we replace logarithmic manifold by locally analytically constructible space everywhere. In particular, Theorem \ref{continuousmapsl2space} holds for the restricted case as long as $\{\hbox{Pure LMHS}\}\subset \Phi$. If we consider the map \eqref{valuativetosl2} in the restricted case:
\begin{equation}\label{restrictedvaluativetosl2}
    \psi|_{D_{\Sigma, \Phi, \mathrm{val}}}: D_{\Sigma, \Phi, \mathrm{val}}^{\#}\rightarrow D_{\mathrm{SL}_2},
\end{equation}
since $\Sigma$ is a weak fan of type $\Phi$, it is enough to prove for any $\sigma$ the map
\begin{equation}\label{restrictedvaluativetosl2local}
    \psi|_{D_{\sigma, \Phi, \mathrm{val}}}: D_{\sigma, \Phi, \mathrm{val}}^{\#}\rightarrow D_{\mathrm{SL}_2, \Phi},
\end{equation}
is continuous, here $D_{\mathrm{SL}_2, \Phi}$ is defined to be the subspace of $D_{\mathrm{SL}_2}$ containing all $\mathrm{SL}_2$ orbits (indeed their equivalent classes) whose underlying nilpotent orbits are of type $\Phi$.  The rest of the proof can be obtained by applying \cite[Sec. 6.4.12]{KU08} verbatim.

Finally, Theorem \ref{kumainthmA'} is obtained by patching up every results we have for the restricted case. Theorem \ref{kumainthmB'} can be seen as follows: Suppose for a period map $\varphi: S\rightarrow \Gamma\backslash D$, $\Phi$ is the index set containing all LMHS types accessed by $\varphi$. It is enough to check on local charts, i.e.,
\begin{equation}
    (\Delta^{*})^k\times \Delta^l\rightarrow \Gamma(\sigma)\backslash D_{\sigma, \Phi}
\end{equation}
is a morphism of locally analytically constructible spaces for any $\sigma\in \Sigma$, but this is immediate as 
\begin{equation}
    (\Delta^{*})^k\times \Delta^l\rightarrow \Gamma(\sigma)\backslash D_{\sigma}
\end{equation}
is a morphism of logarithmic manifolds with image in $\Gamma(\sigma)\backslash D_{\sigma, \Phi}$, and 
\begin{equation}
    \Gamma(\sigma)\backslash D_{\sigma, \Phi}\hookrightarrow \Gamma(\sigma)\backslash D_{\sigma}
\end{equation}
is clearly a morphism of locally analytically constructible spaces. This concludes the proof.